\documentclass[12pt,a4paper]{amsart}

\usepackage{txfonts}

\usepackage{hyperref}
\usepackage{latexsym}
\usepackage{amssymb}

\hypersetup{
    colorlinks=false,
    pdfborder={0 0 0},
}

\newtheorem{theorem}{Theorem}[section]
\newtheorem{lemma}[theorem]{Lemma}
\newtheorem{corollary}[theorem]{Corollary}

\theoremstyle{definition}

\newtheorem{remark}[theorem]{Remark}

\numberwithin{equation}{section}

\begin{document}

\title[On a geometric inequality related to fractional integration]{{\bf On a geometric  inequality related to fractional integration}}
\author{\small TING CHEN}
\address{Ting Chen:   School of Mathematics, University of Edinburgh, EH9 3JZ, UK.}
 \email{t.chen-16@sms.ed.ac.uk}

\footnotetext{\hspace{-3pt}
}

\begin{abstract}
 In this paper we consider a new kind of inequality  related to fractional integration, motivated by Gressman's paper. Based on it we investigate its multilinear analogue inequalities. Combining with the Gressman's work on multilinear integral, we establish this new kind of geometric  inequalities with bilinear form and multilinear form in more general settings. Moreover, in some cases we also find the best constants and optimisers for these geometric inequalities on Euclidean spaces with Lebesgue measure settings with $L^{p}$ bounds.
\end{abstract}

\maketitle

\section{Introduction}

Several  fractional  integral  inequalities have been studied. The Hardy-Littlewood-Sobolev inequality asserted
$$\| \int_{\mathbb{R}} g(t)  \frac{1}{|s-t|^{2-\frac{1}{p}-\frac{1}{q}}} dt\|_{p^{\prime}} \leq C_{p,q} \ \| g\|_{q},$$
for $1< p$, $q< \infty$,  $\frac{1}{p}+\frac{1}{q} > 1$ and all  functions $g$ in $L^{q}(\mathbb{R})$.
Applying H\"{o}lder's inequality gives the forward Hardy-Littlewood-Sobolev inequality $(1.1)$:
for $1< p$, $q< \infty$,  $\frac{1}{p}+\frac{1}{q} > 1$
$$| \int_{\mathbb{R}}  \int_{\mathbb{R}}  f(s) g(t) \frac{1}{|s-t|^{2-\frac{1}{p}-\frac{1}{q}}} ds \ dt| \leq C_{p,q} \ \| f\|_{p}
\| g\|_{q}  \eqno(1.1)$$
holds for  all functions $f \in L^{p}(\mathbb{R})$ and $g \in L^{q}(\mathbb{R})$.
Gressman  \cite{Gressman}  showed the equivalence between the forward Hardy-Littlewood-Sobolev inequality $(1.1)$ and
the inverse Hardy-Littlewood-Sobolev inequality $(1.2)$ which follows. For $0< p$, $q< 1$,
and all functions $f \in L^{p}(\mathbb{R})$ and $g \in L^{q}(\mathbb{R})$, we have
$$\|f\|_{p} \| g\|_{q} \leq C_{p,q}  \int_{\mathbb{R}}  \int_{\mathbb{R}}  | f(s) g(t)|  |s-t|^{\frac{1}{p}+\frac{1}{q}-2}  dsdt .   \eqno(1.2)$$
As a result, for $0 < p$,  $q < r <\infty$  and all measurable functions $f,  g$,
$$\|f^{r}\|_{\frac{p}{r}} \| g^{r}\|_{\frac{q}{r}} \leq C_{\frac{p}{r},\frac{q}{r}}
\int_{\mathbb{R}}  \int_{\mathbb{R}} |  f^{r}(s) g^{r}(t)|  |s-t|^{\frac{r}{p}+\frac{r}{q}-2}  ds dt .$$
Then,  for $0 < p$,  $q < r <\infty$  and all functions $f \in L^{p}(\mathbb{R})$ and $g \in L^{q}(\mathbb{R})$,  we have
$$\|f\|_{p}\  \|g\|_{q} \leq
 C_{\frac{p}{r},\frac{q}{r}}^{\frac{1}{r}} \ \| f(s) g(t) |s-t|^{\frac{1}{p}+\frac{1}{q}-\frac{2}{r}} \|_{L^{r}(dsdt)}. \eqno(1.3)$$
It is natural to ask what the inequality $(1.3)$ leads to if we let $r \rightarrow \infty$,
and what happens to the constant $C_{\frac{p}{r},\frac{q}{r}}^{\frac{1}{r}}$  as  $r \rightarrow \infty$.
One approach is to consider the behaviour of  $C_{\frac{p}{r},\frac{q}{r}}^{\frac{1}{r}}$ as $r \rightarrow \infty$,
but as we are interested in geometrical questions, we prefer a more direct approach. \\

Thus, we will be asking the following questions mainly motivated by the multilinear fractional integrals in Gressman's paper \cite{Gressman}.

\medskip

\hspace{-13pt}{\bf Question 1.} Let $f, g$ be measurable functions defined on $\mathbb{R}$ with Lebesgue measure.
Does there exist $C_{p,q}$ such that for any $0 < p, q <\infty$, $\gamma = \frac{1}{p}+\frac{1}{q}$
$$\|f\|_{p}\  \|g\|_{q} \leq C_{p,q} \ \displaystyle{\sup_{s,t}} \  | f(s)g(t)||s-t|^{\gamma}?  \eqno(1.4)$$

The reason why we take $\gamma = \frac{1}{p}+\frac{1}{q}$ follows from homogeneity.
We assume that $\|f\|_{p}\  \|g\|_{q} \leq C_{p,q} \  \displaystyle{\sup_{s,t}} \  |  f(s)g(t)||s-t|^{\gamma}$  holds. Then consider functions
$f(\frac{\cdot}{R})$, $g(\frac{\cdot}{R})$ for all $R >0$:
$$\|f(\frac{\cdot}{R})\|_{p} = R^{\frac{1}{p}} \|f\|_{p}, \  \|g(\frac{\cdot}{R})\|_{q} = R^{\frac{1}{q}} \|g\|_{p},$$
and
\begin{align*}
\displaystyle{\sup_{s,t}} \  |  f(\frac{s}{R})g(\frac{t}{R})||s-t|^{\gamma}
&= R^{\gamma}  \displaystyle{\sup_{s,t}} \  | f(\frac{s}{R})g(\frac{t}{R})||\frac{s}{R}-\frac{t}{R}|^{\gamma}  \\
&= R^{\gamma} \displaystyle{\sup_{s,t}} \  | f(s)g(t)||s-t|^{\gamma}  .
\end{align*}
So  $$R^{\frac{1}{p}+\frac{1}{q}} \|f\|_{p}\  \|g\|_{p} \leq C_{p,q} \  R^{\gamma} \displaystyle{\sup_{s,t}} \  | f(s)g(t)||s-t|^{\gamma}.$$
This indicates for all $R > 0$,
$$R^{\frac{1}{p}+\frac{1}{q}} \leq  C_{p,q} \  R^{\gamma}$$
which implies $\gamma = \frac{1}{p}+\frac{1}{q}$.

If we consider the simple case when $f$, $g$ are supported in an interval $E \subset \mathbb{R}$, we find
$$\|f\|_{p}\  \|g\|_{q} \leq \|f\|_{\infty} \  \|g\|_{\infty} |E|^{\frac{1}{p} + \frac{1}{q}}
 =  \displaystyle{\sup_{s}} |f(s)|  \ \displaystyle{\sup_{t}} |g(t)| \  \displaystyle{\sup_{s,t}} |s-t|^{\frac{1}{p}+\frac{1}{q}},$$
where $|E|$ is the Lebesgue measure of $E$.
Clearly the right side of (1.4) is in principle smaller than  $\displaystyle{\sup_{s}} |f(s)| \ \displaystyle{\sup_{t}} |g(t)| \ \displaystyle{\sup_{s,t}} |s-t|^{\frac{1}{p}+\frac{1}{q}}$.

We establish that the answer is positive as can be seen in Section 2. More precisely, we prove that if
$f \in L^{p}(\mathbb{R}^{n}), g \in L^{q}(\mathbb{R}^{n})$ then
$$\|f\|_{ L^{p}(\mathbb{R}^{n})}\  \|g\|_{ L^{q}(\mathbb{R}^{n})} \leq C_{p,q,n} \ \displaystyle{\sup_{s,t}} \  | f(s)g(t)||s-t|^{\gamma},$$
for any $0 < p,q <\infty$, $\gamma= \frac{n}{p}+\frac{n}{q}$.

\bigskip

\hspace{-13pt}{\bf Question 2.} What are the analogues of inequality $(1.4)$ in more general settings?

 We investigate what inequality (1.4) would be like in general metric space with a certain geometric measure as shown in Theorem 2.1.

\bigskip

\hspace{-13pt}{\bf Question 3.} Furthermore based on Question 1 and 2,  what are the  multilinear analogues of inequality (1.4)?

Below we give two possible multilinear versions of inequality (1.4). \\
Firstly, let $f_{j}$ be measurable functions defined on $\mathbb{R}^{n}$ with Lebesgue measure.
Does there exist a finite constant $C$ independent of functions $f_{j}$ such that the following  multilinear geometric inequality  (1.5) holds for any $0 < p_{j} <\infty$,
$j=1, \dots, n+1$, $\gamma=\displaystyle{ \sum_{j=1}^{n+1} } \frac{1}{p_{j}}$ ,
$$  \displaystyle{\prod_{j=1}^{n+1}} \|f_{j}\|_{p_{j}}
\leq C \ \displaystyle{\sup_{y_{j}}}  \displaystyle{\prod_{j=1}^{n+1}} | f_{j}(y_{j})| \det(y_{1}, \dots, y_{n+1})^{\gamma}?  \eqno(1.5)$$
The condition $\gamma=\displaystyle{ \sum_{j=1}^{n+1} } \frac{1}{p_{j}}$   follows from homogeneity.  \\
Here the notation $\det(y_{1}, \dots, y_{n+1})$
denotes $n!$ times the Euclidean $n$-dimensional volume of the simplex with vertices $y_{1}, \dots, y_{n+1}$,
so $\det(y_{1}, \dots, y_{n+1}) \geq 0$ throughout the paper.

Furthermore, combining with  Gressman's work  \cite{Gressman}  we investigate what inequality (1.5) would be like in more general settings apart from in the Euclidean space cases , for instance, in a real  finite-dimensional Hilbert space $H$ with a certain geometric measure as  discussed in  \cite{Gressman}.

\medskip

The second possible multilinear form we study is to replace the determinant form by ``product form" as follows.

Let $f_{j}$  be measurable functions defined on $\mathbb{R}^{n}$, $ r_{12},  r_{13},  r_{23}>0 $.
Does there exist a finite constant $C$ independent of the functions $f_{j}$ such that
for any $0< p_{j}< \infty$ satisfying $\displaystyle{ \sum_{j=1}^{3}   }  \frac{1}{p_{j}}= \frac{1}{n} ( r_{12} +r_{13}+ r_{23})$ ,
$$ \|f_{1}\|_{p_{1}}  \|f_{2}\|_{p_{2}}  \|f_{3}\|_{p_{3}}
\leq C \ \displaystyle{\sup_{y_{j}}}  \prod_{j=1}^{3}  f_{j}(y_{j})   |y_{1}-y_{2}|^{r_{12}}  |y_{1}-y_{3}|^{r_{13}}   |y_{2}-y_{3}|^{r_{23}} ? \eqno(1.6)$$
Homogeneity requires   $\displaystyle{ \sum_{j=1}^{3}   }  \frac{1}{p_{j}}= \frac{1}{n} ( r_{12} +r_{13}+ r_{23})$.

\bigskip

\hspace{-13pt}{\bf Question 4.} Do there exist sharp versions and optimisers for these geometric inequalities above?

The main results under this heading are Theorem 4.1 and Theorem 4.6. \\

The purpose of this paper is to study these  geometric  inequalities related to fractional integration in these questions above.
The results can be discussed  as follows.
Section 2 is devoted to studying the  bilinear geometric inequalities raised in Question 1 and 2,
and the main results are established in Theorem 2.1 and Corollary 2.3.
In Section 3, we give two analogues of multilinear form, such as the determinant form as shown in Theorem 3.1, Theorem 3.3
and the product form shown in Theorem 3.5, Theorem 3.7.
In Section 4, in the  Euclidean space setting  we prove the existence of extremal functions for the geometric inequality bilinear form (1.4) when $p=q$
and for the multilinear form (1.5)  when $p_{j}=p, j=1, \dots , n+1$.
Meanwhile, we get the corresponding conformally equivalent formulations in unit sphere space $\mathbb{S}^{n}$  and  in  hyperbolic space $\mathbb{H}^{n}$.

Throughout this paper $\sup$ is the essential supremum of function. $|\cdot|$ denotes the Lebesgue measure on Euclidean space $\mathbb{R}^{n}$  and the norm in a Hilbert space.
$A \lesssim B$ means there exists  a positive constant $C$ independent of the essential variables such that $A \leq C B$.
$A \sim B$ means  there exist  positive constants $C, C^{\prime}$  independent of functions such that $C^{\prime}  B \leq A \leq C B$.
For the rest of this paper, all functions considered are nonnegative.

\bigskip

\section{Bilinear forms of geometric inequalities}

Let  $(M, \ d)$ be a metric space  and  $\mu$  a $\sigma$-finite nonnegative Borel measure on $M$.
Let  $f, g$ be nonnegative measurable functions defined on $M$.
To answer Question 1, 2 we consider the two conditions:

\bigskip

(i)\  For any $x \in M, r>0$, $\mu$ satisfies
$$\mu(B(x, r))\leq C_{\alpha} r^{\alpha} \eqno(2.1)$$
 with a finite constant $C_{\alpha}$, $ \alpha >0$ .

(ii)\ $$\|f\|_{L^{p}(d \mu)}\  \|g\|_{L^{q}(d \mu)} \leq  C_{p,q, \gamma}
\displaystyle{\sup_{s,t}} \ f(s)g(t)d(s, t)^{\gamma}  \eqno(2.2)$$
holds for all nonnegative functions $f \in L^{p}(d \mu), g \in L^{q}(d \mu)$
with a finite constant $C_{p,q,\gamma}$  independent of the funtions $f, g$.

\bigskip

The main results are as follows.

\begin{theorem}
Let $(M, \ d)$ be a metric space and $\mu$ a $\sigma$-finite, nonnegative Borel measure on $M$.

(a)\ If condition (i) holds,
then (ii) holds for all nonnegative functions $f \in L^{p}(d \mu), g \in L^{q}(d \mu)$ for all $0 < p, q <\infty$, $\gamma$ such that $\gamma= \alpha(\frac{1}{p}+\frac{1}{q})$.

(b)\ If condition (ii) holds for all nonnegative functions $f \in L^{p}(d \mu), g \in L^{q}(d \mu)$ for some $p, q>0, \gamma >0$,
then condition (i) holds for all $\alpha$ such that $\alpha= \gamma (\frac{1}{p}+\frac{1}{q})^{-1}$.

\end{theorem}

We begin by studying an endpoint case of (2.2)  in  the following Lemma 2.2,  before studying Theorem 2.1  itself.

\begin{lemma}
Let $f, g$  be nonnegative measurable functions defined on a metric space $(M, \ d)$ with the $\sigma$-finite and nonnegative Borel measure $\mu$
which satisfies $\mu(B(x, r))\leq C_{\alpha} r^{\alpha}$ for any $x \in M, r>0$.
Then  for  all $0 < p, q <\infty$  we have
$$\|f\|_{L^{p, \infty}(d \mu) }   \|g\|_{L^{\infty} (d \mu) } \leq  C_{p, \alpha } \displaystyle{\sup_{s,t}} \ f(s)g(t)d(s, t)^{\frac{\alpha}{p}}.  \eqno(2.3)  $$
$$\|f\|_{L^{\infty} (d \mu) }\  \|g\|_{L^{q, \infty}(d \mu)}  \leq C_{q, \alpha }\displaystyle{\sup_{s,t}} \ f(s)g(t)d(s, t)^{\frac{\alpha}{q}}.  \eqno(2.4) $$

\end{lemma}

\begin{proof}
If $\displaystyle{\sup_{s,t}} \  f(s)g(t)d(s, t)^{\frac{\alpha}{p}}  = \infty$, then the inequality (2.3) is trivial. \\
If $\displaystyle{\sup_{s,t}} \  f(s)g(t)d(s, t)^{\frac{\alpha}{p}}  = A < \infty$,
there exists  a measure zero set $E \subset  M \times M$, $\mu \otimes \mu (E)=0$, such that for any $(s, t) \in (M \times M )\setminus E$
$$f(s)g(t)d(s, t)^{\frac{\alpha}{p}} \leq A.  \eqno(2.5) $$
Note that for any $\varepsilon > 0$, there exists  $F \subset  M$, $\mu (F) > 0$, such that for all $t \in F$
$$ g(t) > \|g\|_{L^{\infty}(d \mu)}-\varepsilon.   \eqno(2.6)$$
It follows from (2.5) and (2.6)  that  for all $(s, t) \in (M \times F) \setminus E$
$$f(s) \leq \frac{A}{d(s, t)
^{\frac{\alpha}{p}} (\|g\|_{\infty} - \varepsilon ) }.$$
So we can choose a $t \in F$ such that for any $\beta > 0$,
$$ \mu (\{ s: f(s) > \beta \}) \leq  \mu (\{ s: d(s, t)^{\frac{\alpha}{p}} < \frac{A}{\beta (\|g\|_{\infty}-\varepsilon )} \} ).$$
This is because $\mu \otimes \mu (E)=0$ implies that for almost every $t \in M$,
$$\mu( \{s \in M: (s, t) \in E\})=0.$$
And since  $\mu (F) >0$,   we can find $t \in F$ such
that $(s, t) \in  (M \times F) \setminus E$ for almost every $s \in M$.

Calculate  the weak $ L^{p}$  ``norm" of $f$,
\begin{align*}
\|f\|_{L^{p, \infty}(d \mu)}
&= \displaystyle{\sup_{\beta> 0}} \ \beta \ \mu (\{ s: f(s) > \beta \} )^{\frac{1}{p}}  \\
&\leq \displaystyle{\sup_{\beta > 0}} \ \beta  \ \mu ( \{ s: d(s, t)^{\frac{k \alpha}{p}} < \frac{A}{\beta (\|g\|_{\infty}-\varepsilon )} \} )^{\frac{1}{p}} \\
&= \displaystyle{\sup_{\beta > 0}} \ \beta \ \mu ( \{ s: d(s, t) < (\frac{A}{\beta (\|g\|_{\infty}-\varepsilon )} )^{ \frac{p}{\alpha} }  \} )^{\frac{1}{p}}.
\end{align*}
Since $\mu(B(x, r))\leq C_{\alpha}  r^{\alpha}$ for any $x\in M, r>0$,
\begin{center}
$ \mu (\{s: d(s, t) < r \} ) \leq C_{\alpha} r^{\alpha}$.
\end{center}
Hence
$$\mu ( \{ s: d(s, t) < (\frac{A}{\beta (\|g\|_{\infty}-\varepsilon )} )^{ \frac{p}{\alpha} }  \} )
\leq  C_{\alpha} (\frac{A}{\beta (\|g\|_{\infty}-\varepsilon )} )^{p}.$$
Then we get
\begin{align*}
\|f\|_{L^{p, \infty}(d \mu)}
&\leq \displaystyle{\sup_{\beta > 0}} \ \beta \ \mu ( \{ s: d(s, t) < (\frac{A}{\beta (\|g\|_{\infty}-\varepsilon )} )^{ \frac{p}{\alpha} }  \} )^{\frac{1}{p}} \\
&\leq C_{\alpha}^{\frac{1}{p}}
\displaystyle{\sup_{\beta > 0}} \ \beta \ \frac{A}{\beta (\|g\|_{\infty}-\varepsilon )}  \\
&= C_{\alpha}^{\frac{1}{p}} \frac{A}{ \|g\|_{\infty}-\varepsilon },
\end{align*}
that is
$$\|f\|_{L^{p, \infty}(d \mu) } ( \|g\|_{L^{\infty}(d \mu)} - \varepsilon  )  \leq C_{\alpha}^{\frac{1}{p}}  A. $$
Let $\varepsilon  \rightarrow 0$, we have
$$\|f\|_{L^{p, \infty}(d \mu) }   \|g\|_{L^{\infty} (d \mu) }   \leq C_{\alpha}^{\frac{1}{p}}  A = C_{\alpha}^{\frac{1}{p}}  \displaystyle{\sup_{s,t}} \  f(s)g(t)d(s, t)^{\frac{\alpha}{p}}.$$
Likewise,
$$\|f\|_{L^{\infty} (d \mu) }\  \|g\|_{L^{q, \infty}(d \mu)}  \leq
C_{\alpha}^{\frac{1}{q}}  \displaystyle{\sup_{s,t}} \  f(s)g(t)d(s, t)^{\frac{\alpha}{q}}.$$
\end{proof}

\bigskip

\hspace{-13pt}{\it Proof of Theorem 2.1}\quad

(a)\ Suppose condition (i) holds, that is $\mu(B(t, r))\leq C r^{\alpha}$ holds for any $t \in M, r>0$.
Let $m= \frac{1}{\frac{1}{p}+\frac{1}{q}}$, so $m < p, \ q <\infty$. \\
Then by the layer cake representation
\begin{align*}
\|f\|_{L^{p}(d \mu)}^{p} &=
p \int_{0}^{\infty} \beta^{p-1} \mu (\{s: f(s) > \beta \} ) \ d\beta \\
&= p \int_{0}^{\|f\|_{L^{\infty}(d \mu)}} \beta^{p-1} \mu ( \{s: f(s) > \beta \} ) \ d\beta +
p \int_{\|f\|_{L^{\infty}(d \mu)}}^{\infty} \beta^{p-1} \mu( \{s: f(s) > \beta \} ) \ d\beta \\
&= p \int_{0}^{\|f\|_{L^{\infty}(d \mu)}} \beta^{p-m-1} \beta^{m} \mu ( \{s: |f(s)| > \beta \} ) \ d\beta \\
&\leq p \ \|f\|_{L^{m, \infty}(d \mu)}^{m} \int_{0}^{\|f\|_{L^{\infty}(d \mu)}} \beta^{p-m-1} \ d\beta \\
&= \frac{p}{p-m} \ \|f\|_{L^{m, \infty}(d \mu)}^{m} \|f\|_{L^{\infty}(d \mu)}^{p-m},
\end{align*}
which means for $f$ in $L^{m, \infty}(d \mu)\cap L^{\infty}(d \mu)$, we have $f \in L^{p}(d \mu)$, and
$$\|f\|_{L^{p}(d \mu)} \leq (\frac{p}{p-m})^{\frac{1}{p}} \ \|f\|_{L^{m, \infty}(d \mu)}^{\frac{m}{p}} \|f\|_{L^{\infty}(d \mu)}^{1-\frac{m}{p}}.  \eqno(2.7)$$
Meanwhile if $g$ is in $L^{m, \infty}(d \mu)\cap L^{\infty}(d \mu)$, then $g \in L^{q}(d \mu)$, and
$$\|g\|_{L^{q}(d \mu)} \leq (\frac{q}{q-m})^{\frac{1}{q}} \ \|g\|_{L^{m, \infty}(d \mu)}^{\frac{m}{q}} \|g\|_{L^{\infty}(d \mu)}^{1-\frac{m}{q}}.  \eqno(2.8)$$

\bigskip

Since simple functions are in $L^{m, \infty}(d \mu) \cap L^{\infty}(d \mu)$, we can apply Lemma 2.2 for simple functions $f, g$.
Inequalities  (2.7) and (2.8) indicate
\begin{align*}
\|f\|_{L^{p}(d \mu)}\  \|g\|_{L^{q}(d \mu)} &\leq (\frac{p}{p-m})^{\frac{1}{p}} \ (\frac{q}{q-m})^{\frac{1}{q}}
\|f\|_{m, \infty}^{\frac{m}{p}} \|f\|_{\infty}^{1-\frac{m}{p}} \|g\|_{m, \infty}^{\frac{m}{q}}   \|g\|_{\infty}^{1-\frac{m}{q}} \\
&=  (\frac{p}{p-m})^{\frac{1}{p}} \ (\frac{q}{q-m})^{\frac{1}{q}}
( \|f\|_{m, \infty}^{\frac{m}{p}}   \|g\|_{\infty}^{1-\frac{m}{q}} )  ( \|g\|_{m, \infty}^{\frac{m}{q}}    \|f\|_{\infty}^{1-\frac{m}{p}} ).
\end{align*}
It follows from Lemma 2.2 that
 $$\|f\|_{m, \infty}^{\frac{m}{p}}   \|g\|_{\infty}^{1-\frac{m}{q}}
 =  \|f\|_{m, \infty}^{\frac{m}{p}}   \|g\|_{\infty}^{\frac{m}{p}}
 \leq C_{\alpha}^{\frac{1}{p}}   \displaystyle{\sup_{s,t}}  \ ( f(s)g(t)d(s, t)^{\frac{\alpha}{m}} )^{\frac{m}{p}},$$
and
$$\|g\|_{m, \infty}^{\frac{m}{q}}    \|f\|_{\infty}^{1-\frac{m}{p}} =  \|g\|_{m, \infty}^{\frac{m}{q}}    \|f\|_{\infty}^{\frac{m}{q}}
 \leq C_{\alpha}^{\frac{1}{q}}  \displaystyle{\sup_{s,t}} \  ( f(s)g(t)d(s, t)^{\frac{\alpha}{m}} )^{\frac{m}{q}}.$$
Therefore
\begin{align*}
\|f\|_{L^{p}(d \mu)}\  \|g\|_{L^{q}(d \mu)}
&\leq C_{\alpha}^{\frac{1}{p}+\frac{1}{q}} (\frac{p}{p-m})^{\frac{1}{p}} \ (\frac{q}{q-m})^{\frac{1}{q}}
 \displaystyle{\sup_{s,t}} ( f(s)g(t)d(s, t)^{\frac{\alpha}{m}} ) ^{\frac{m}{p}}
\displaystyle{\sup_{s, t}} ( f(s)g(t)d(s, t)^{\frac{\alpha}{m}} ) ^{\frac{m}{q}} \\
&=  C_{\alpha}^{\frac{1}{p}+\frac{1}{q}} (\frac{p+q}{p})^{\frac{1}{p}} (\frac{p+q}{q})^{\frac{1}{q}}
\displaystyle{\sup_{s, t}} \ f(s)g(t)d(s, t)^{\frac{\alpha}{p}+\frac{\alpha}{q}}.
\end{align*}

For  general functions $f\in L^{p}(d \mu),  g\in L^{q}(d \mu)$, there exist sequences of  simple functions  $\{f_{n}\} \uparrow f$, and
$\{g_{n}\} \uparrow g$ as $n \rightarrow \infty$.
Under the discussion above, we have already obtained that (2.2) holds for simple functions,
$$\|f_{n}\|_{L^{p}(d \mu)}\  \|g_{n}\|_{L^{q}(d \mu)} \leq C_{p,q,\alpha}
\displaystyle{\sup_{s,t}} \ f_{n}(s) g_{n}(t) d(s, t)^{ \alpha(\frac{1}{p}+\frac{1}{q})}
\leq C_{p,q,\alpha} \displaystyle{\sup_{s, t}} \ f(s)g(t)d(s, t)^{\alpha(\frac{1}{p}+\frac{1}{q})}.$$
Then let $n \rightarrow \infty$, we have
$$\|f\|_{L^{p}(d \mu)}\  \|g\|_{L^{q}(d \mu)} \leq C_{p,q,\alpha} \displaystyle{\sup_{s, t}} \ f(s)g(t)d(s, t)^{ \alpha(\frac{1}{p}+\frac{1}{q})}.$$
(b)\ Suppose $\|f\|_{L^{p}(d \mu)}\  \|g\|_{L^{q}(d \mu)} \leq C_{p,q, \gamma}
\displaystyle{\sup_{s,t}} \ f(s)g(t)d(s, t)^{\gamma}$ holds for some
$p, q >0$, $\gamma$. For any $x \in M, r >0$, let $f=g=\chi_{B(x, r)}$, then we have
$$\mu (B(x, r))^{\frac{1}{p}+\frac{1}{q}}
\leq C_{p,q, \gamma} \displaystyle{\sup_{s, t \in B(x, r)}} d(s, t)^{ \gamma}.$$
Together with the fact
$$\displaystyle{\sup_{s, t \in B(x, r)}} d(s, t) \leq 2r $$
we deduce that $\mu$ has the  property
$$\mu (B(x, r)) \leq C_{\alpha} r^{\alpha},$$
where $\alpha= \gamma (\frac{1}{p}+\frac{1}{q})^{-1}$.

$\hfill\Box$

\bigskip

\begin{corollary}
Let $f, g$  be measurable functions defined on $\mathbb{R}^{n}$ with Lebesgue measure, then
for all $0 < p, q <\infty$, $\gamma >0$ such that $\gamma = n(\frac{1}{p}+\frac{1}{q})$,
$$\|f\|_{L^{p}(\mathbb{R}^{n}) }   \|g\|_{L^{q} (\mathbb{R}^{n})}  \leq C_{p,q,n}  \displaystyle{\sup_{s,t}} \ f(s)g(t)|s-t|^{\gamma} .   \eqno(2.9)$$
On the other hand, it is not true that
$$\|f\|_{L^{p}(\mathbb{R}^{n}) }   \|g\|_{L^{\infty} (\mathbb{R}^{n}) }  \lesssim   \displaystyle{\sup_{s,t}} \ f(s)g(t)|s-t|^{\frac{n}{p}}.  \eqno(2.10)  $$
holds  for all $ f\in L^{p}(\mathbb{R}^{n}), g \in L^{\infty} (\mathbb{R}^{n})$.
\end{corollary}

\begin{proof}
(1) \ Observe that
$|B(x, r)| \leq C_{n} r^{n}$  for any $x\in \mathbb{R}^{n}, r>0$,
then we can apply Theorem 2.1 to give (2.9).
More precisely, $\gamma$ here must be $n(\frac{1}{p}+\frac{1}{q})$ which follows from the homogeneity mentioned in the introduction.

(2) \  We use a counterexample to show that  (2.10) fails. \\
For any positive $N$,  let $f_{N}(s)= (1+|s|)^{-\frac{n}{p}} \chi_{(1 \leq |s| \leq N)}$, $g(t)= \chi_{ (|t| \leq 1)}(t)$.
Then $\|g\|_{L^{\infty} (\mathbb{R}^{n}) }=1$ and
\begin{align*}
\displaystyle{\sup_{s,t}} \  \ f_{N}(s)g(t) \ |s-t|^{\frac{n}{p}}
&= \displaystyle{\sup_{s,t}} \ \frac{|s-t|^{\frac{n}{p}}}{(1+|s|)^{\frac{n}{p}}} \ \chi_{(1 \leq |s| \leq N)}(s)  \  \chi_{ (|t| \leq 1)}(t) \\
&\leq \frac{(|s|+1)^{\frac{n}{p}}}{(1+|s|)^{\frac{n}{p}}} =1.
\end{align*}
While by polar coordinates
$$\|f_{N}\|_{L^{p}(\mathbb{R}^{n}) }^{p}= \int_{1 \leq |s| \leq N} \ \frac{ds}{(1+|s|)^{n}}
=C \int_{1}^{N}   \frac{r^{n-1}}{(1+r)^{n}} dr.  $$
Let $ u=1+r$ to  make the  change of  variables
$$  \int_{1}^{N}   \frac{r^{n-1}}{(1+r)^{n}} dr
= \int_{2}^{N+1}   \frac{(u-1)^{n-1}}{u^{n}} du
\geq  \int_{2}^{N+1}   \frac{1}{2^{n-1} }  \frac{1}{u}  du $$
and
$$\int_{2}^{N+1}   \frac{1}{u}  du = \ln (N+1) -\ln 2 \rightarrow \infty,$$
as $N \rightarrow \infty$.

\end{proof}

\bigskip

\section{Multilinear forms of geometric inequalities}

We first recall some terminology, notation and lemmas which are all given in  \cite{Gressman}.
$(H, \langle \cdot, \cdot \rangle_{H})$ is a  real  finite-dimensional Hilbert space with inner product  $\langle \cdot, \cdot \rangle_{H}$.
For any positive integer $k \leq \mathrm{dim} \  H$,  we use $\det(y_{1}, \dots, y_{k+1})$ to denote the square root of the determinant
of the $k \times k$  Gram matrix $(a_{i,j})_{k \times k}$, where
$$a_{i,j} =\langle y_{i}-y_{k+1},  y_{j}-y_{k+1} \rangle_{H}.$$

Clearly, the Gram matrix $(a_{i,j})_{k \times k}$ is positive semidefinite, since
$(a_{i,j})_{k \times k}$ can be written as $A^{\prime} A$, where $ A$  is the matrix whose $ j$-th column is $ y_{j}-y_{k+1}$,
and $A^{\prime}$ is the transpose of $A$.
So $\det(y_{1}, \dots, y_{k+1}) \geq 0$ throughout the paper.

Especially  in Euclidean  $\mathbb{R}^{k}$ space,  the determinant of the matrix $(a_{i,j})_{k \times k}$ is the square of
the volume of  the parallelotope formed by the vectors $y_{1}, \dots, y_{k+1}$.
Thus, $\det(y_{1}, \dots, y_{k+1})$ is also
$k!$ times the Euclidean $k$-dimensional volume of the simplex with vertices $y_{1}, \dots, y_{k+1}$. \\

\hspace{-13pt}{\bf Definition 1.}  A subset $B \subset H$ is called an ellipsoid when it may be written as
$$ B \equiv \{x\in H:  \displaystyle{\sum_{i}} \frac{ | \langle x-x_{0}, \omega_{i} \rangle |^{2}}{l_{i}^{2}} \leq 1  \}$$
for some $x_{0}\in H$, some orthonormal basis $\{\omega_{i}\}$ of $H$, and lengths $l_{i}\in [0, \infty]$.
For example, $\{(t,0, \dots, 0): t\in \mathbb{R}\} \subset \mathbb{R}^{n}$ is an ellipsoid in  $\mathbb{R}^{n}$.
It could be  written as
$$\frac{ | \langle x, e_{1} \rangle |^{2}}{\infty}+ \frac{ | \langle x, e_{2} \rangle |^{2}}{0}+\cdots+ \frac{ | \langle x, e_{n} \rangle |^{2} }{0} \leq 1 ,$$
where $ l_{1}= \infty,    l_{2}=0, \dots, l_{n}=0$ , $x_{0}=0$,  and $\{ e_{1}, \dots, e_{n}\}$ are the standard orthonormal basis  vectors for $ \mathbb{R}^{n}$.
The  ellipsoid  will be called centred when $x_{0}=0$.
Given an ellipsoid $B \subset H$ and an integer $k$ with $k \leq $ dim $H$,
denote
$$|B|_{k}=\sup \{l_{i_{1}} \dots l_{i_{k}}: i_{1}< i_{2}< \dots <i_{k} \},$$
 which is called the $k$-content of $B$. \\

\hspace{-13pt}{\bf Definition 2.} A $\sigma$-finite and nonnegative Borel measure $\mu$ is called $k$-curved with exponent $\alpha > 0$, if there exists a finite
 constant $C_{\alpha}$ such that
$$\mu(B) \leq C_{\alpha} |B|_{k}^{\alpha} \eqno (3.1)$$
 for all ellipsoids $B$ in $H$.

\bigskip

This kind of geometric measure describes the amount of mass of $\mu$ supported on $k$-dimensional subspaces of $H$.
For instance, the Lebesgue measure  in $\mathbb{R}^{n}$ is $n$-curved with exponent $1$.
It is $k$-curved with exponent $\frac{n}{k}$ as well for $k <n$.
If we see the Lebesgue measure restricted on $x_{1}$ axis, it is $1$-curved with exponent $1$.
It cannot  be $k$-curved for $k \geq 2$.
Let $S$ be a hypersurface in $\mathbb{R}^{n}$ with non-vanishing Gaussian curvature, then
its surface area measure $\mu_{S}$ is $n$-curved with exponent $\frac{n-1}{n+1}$. \\

We now recall some results of Gressman in  \cite{Gressman}. \\

\hspace{-13pt}{\bf Lemma 3.} [1]  Let $\mu$ be a $\sigma$-finite and nonnegative Borel measure such that  (3.1) holds for all ellipsoids $B$ in $H$.
Then for any measurable sets $E_{1}, \dots, E_{k}$ in $H$ we have
$$\mu \otimes \cdots \otimes \mu (\{(y_{1}, \dots, y_{k})\in E_{1} \times \cdots \times E_{k}:
\det(0, y_{1}, \dots, y_{k}) < \delta \})
\leq C_{k, \alpha}  \delta^{\alpha} \displaystyle{\prod_{j=1}^{k}} \mu(E_{j})^{1-\frac{1}{k}}.$$

\hspace{-13pt}{\bf Lemma 4.} [1]  Under the above assumptions,  for any centred ellipsoid $B$ in $H$, we have
$$\displaystyle{\sup_{x_{j}\in B}} \  \det(0, x_{1}, \dots, x_{k}) \leq  C_{k} |B|_{k},  \eqno(3.2)  $$
where $|B|_{k}$ is the $k$-content of $B$. \\

\hspace{-13pt}{\bf Lemma 5.} [1]   Let $f_{j}$ be nonnegative measurable functions defined on a  real  finite-dimensional Hilbert space $H$, and let
$\mu$ be a $\sigma$-finite  nonnegative Borel measure on $H$ which satisfies inequality (3.1).
Then for all $1 \leq p_{j}  \leq \infty$ satisfying $\frac{1}{p_{j}} > 1- \frac{\gamma}{k \alpha}$, $j=1, \dots, k+1$,
and $k+1 - \displaystyle{ \sum_{j=1}^{k+1} }  \frac{1}{p_{j}} =  \frac{\gamma}{\alpha}$,
 $$ \int_{H} \cdots \int_{H} \displaystyle{\prod_{j=1}^{k+1}} f_{j}(y_{j}) \det(y_{1}, \dots,y_{k+1})^{-\gamma} d\mu(y_{1}) \cdots d\mu(y_{k+1})  \leq C \displaystyle{\prod_{j=1}^{k+1}} \|f_{j}\|_{L^{p_{j}}(d \mu)}      \eqno(3.3)$$
holds with a finite constant $C$ independent of the functions $f_{j}$.

\bigskip

The first kind of  multilinear analogue of the fractional integral inequality we start to study is the determinant form as given in the following theorem,
mainly discussing the two conditions with $1 \leq k \leq \mathrm{dim} \  H$ fixed:

\bigskip

(i)\  There exists a finite  constant $C_{\alpha}$ such that for all  ellipsoids $B$ in $H$,
$$\mu(B) \leq C_{\alpha} |B|_{k}^{\alpha}.   \eqno(3.4)$$

(ii)\ $$ \displaystyle{\prod_{j=1}^{k+1}} \|f_{j}\|_{L^{p_{j}}(d \mu)}
\leq C \displaystyle{\sup_{y_{j}}} \ \displaystyle{\prod_{j=1}^{k+1}} \ f_{j}(y_{j}) \det(y_{1}, \dots, y_{k+1})^{\gamma}  \eqno(3.5)$$
for all nonnegative functions $f_{j} \in L^{p_{j}}(d \mu), j=1, \dots, k+1$, where $C$  is a finite constant
independent of functions $f_{j}$ which only depends on $p_{j}, k, \gamma$ .

\bigskip

\begin{theorem}
Let $(H, \langle \cdot, \cdot \rangle_{H})$ be a  real finite-dimensional Hilbert space. Let  $\mu$  be a  $\sigma$-finite  nonnegative Borel measure .

(a)\ If  condition (i) holds,
then (ii) holds for all nonnegative functions$f_{j} \in L^{p_{j}}(d \mu)$, for all $0 <p_{j}< \infty, \gamma$ which satisfy $ \frac{1}{p_{j}} < \frac{\gamma}{k \alpha}$
and $\displaystyle{ \sum_{j=1}^{k+1} } \ \frac{1}{p_{j}}= \frac{\gamma}{\alpha}$,  $j=1, \dots, k+1$.

(b)\ If condition (ii) holds  for all nonnegative functions$f_{j} \in L^{p_{j}}(d \mu)$,  $j=1, \dots, k+1$, for some $p_{j}>0$,  $\gamma >0$,
then condition (i) holds for all $\alpha$ such that
 $\alpha= \gamma (\displaystyle{ \sum_{j=1}^{k+1} } \ \frac{1}{p_{j}})^{-1}$.

\end{theorem}

\medskip

If we consider the special case when $k=1$, the condition (3.4) is equivalent to the condition (2.1).
It is clear that (3.4) implies (2.1).
Conversely, suppose $\mu(B(x,r)) \leq C_{\alpha} r^{\alpha} $ holds for any $ x\in H, r>0$.
Given an ellipsoid $K$ centred at $x_{0}$, clearly $K \subset  B(x_{0}, |K|_{1})$. So
$$\mu(K) \leq \mu(B(x_{0}, |E|_{1}))  \leq C_{\alpha}|E|_{1}^{\alpha} ,$$
which gives that $\mu$ is $1$-curved with exponent $\alpha$.

When $k=1$, inequality (3.5) becomes the bilinear form (2.2).
In Section 2 we stated that
$$\|f_{1}\|_{L^{p_{1}}(d \mu)}\  \|f_{2}\|_{L^{p_{2}}(d \mu)} \lesssim  \displaystyle{\sup_{s, t} }  \  f_{1}(s) f_{2}(t)|s-t|^{\gamma}$$
holds for any $0 <p_{1},  p_{2}< \infty$, $ \frac{1}{p_{1}}+ \frac{1}{p_{2}} =  \frac{\gamma}{ \alpha}$.
Note that the condition $\frac{1}{p_{j}} < \frac{\gamma}{k \alpha},  j=1, 2$,  in Theorem 3.1 (a)  is automatic in this case,  since $0 <p_{1},  p_{2}< \infty$.

\bigskip

We begin by studying why condition (ii) implies condition (i). \\

\hspace{-13pt}{\it Proof of Theorem 3.1 (b)}\quad

 Let  $f_{j}=\chi_{B}$, where $B \subset H$ is an ellipsoid centred at $x_{0} \in H$,  $ j=1, \dots, k+1$.
Since condition (ii) holds for some $p_{j}, j=1, \dots, k+1, \gamma$, then we have
$$ \displaystyle{\prod_{j=1}^{k+1}} \| \chi_{B} \|_{L^{p_{j}}(d \mu)} \lesssim
\displaystyle{\sup_{y_{j}\in B}} \det(y_{1}, \dots, y_{k+1})^{\gamma},$$
that is,
$$\displaystyle{\prod_{j=1}^{k+1}} \mu (B)^{\frac{1}{p_{j}}} = \mu (B)^{\sum\limits_{j=1}^{k+1} \frac{1}{p_{j}}}
\lesssim \displaystyle{\sup_{y_{j}\in B}} \det(y_{1}, \dots, y_{k+1})^{\gamma}. \eqno(3.6)$$
We use  a fact that for any centred ellipsoid $E$, $E-E \subset 2E$.
Suppose
$$E= \{x\in H:  \displaystyle{\sum_{i}} \frac{ | \langle x, \omega_{i} \rangle |^{2}}{l_{i}^{2}} \leq 1 \}$$
where $\{\omega_{i}\}$ is the  orthonormal basis of $H$.
Let $y, z \in E$, since for every $\omega_{i}$
$$| \langle y-z, \omega_{i} \rangle |^{2}
= | \langle y, \omega_{i} \rangle - \langle z, \omega_{i} \rangle|^{2}
\leq 2(| \langle y, \omega_{i} \rangle |^{2}+ | \langle z, \omega_{i} \rangle|^{2}),$$
it is easy to verify that
$$y-z \in 2E = \{x\in H:  \displaystyle{\sum_{i}} \frac{ | \langle x, \omega_{i} \rangle |^{2}}{(2 l_{i})^{2}} \leq 1 \}.$$
Thus we have
$$B-B= (B-x_{0})-(B-x_{0})\subset 2(B-x_{0}).$$
Therefore, it follows form Lemma 4 that
\begin{align*}
\sup\limits_{y_{j}\in B} \det(y_{1}, y_{2}, \dots, y_{k}, y_{k+1})
&= \sup\limits_{y_{j}\in B} \det(0, y_{1}-y_{k+1} , y_{2}-y_{k+1}, \dots, y_{k}-y_{k+1}) \\
&\leq \sup\limits_{x_{j}\in 2(B-x_{0})} \det(0, x_{1} , x_{2}, \dots, x_{k}) \\
&\leq 2^{k}C_{k} |B-x_{0}|_{k}=2^{k}C_{k} |B|_{k}.
\end{align*}
So
$$\displaystyle{\sup_{y_{j}\in B}} \det(y_{1}, \dots, y_{k+1})^{\gamma}  \lesssim  |B|_{k}^{\gamma}. $$
Together with (3.6), we conclude that
$$\mu (B)^{\sum\limits_{j=1}^{k+1} \frac{1}{p_{j}}} \lesssim \displaystyle{\sup_{y_{j}\in B}} \det(y_{1}, \dots, y_{k+1})^{\gamma}
\lesssim  |B|_{k}^{\gamma}.$$
So $\mu (B) \lesssim  |B|_{k}^{\alpha}$ with $\alpha= \gamma (\displaystyle{ \sum_{j=1}^{k+1} } \ \frac{1}{p_{j}})^{-1}$. \\

$\hfill\Box$

\bigskip

On the other hand,  in order to see what inequality (3.5) will be like if $\mu$  is $k$-curved with exponent $\alpha$,
we first investigate  an endpoint case of (3.5)  as follows.

\begin{lemma}
Let $f_{j}$ be measurable functions defined on real  finite-dimensional Hilbert space $H$ with the $\sigma$-finite and nonnegative Borel measure $\mu$
which satisfies $\mu(B) \leq C_{\alpha} |B|_{k}^{\alpha}$ for all ellipsoids $B \subset H$.\\
Then for any positive $\gamma$ we have
$$ \displaystyle{\prod_{j=1}^{k}} \|f_{j}\|_{L^{\frac{k \alpha}{\gamma}, \infty}(d \mu)} \|f_{k+1}\|_{L^{\infty}(d \mu)} \leq C_{k,\alpha, \gamma }
\displaystyle{\sup_{y_{j}}} \ \displaystyle{\prod_{j=1}^{k+1}} \ f_{j}(y_{j}) \det(y_{1}, \dots, y_{k+1})^{\gamma}.  \eqno(3.7)  $$
Likewise for each $1 \leq l \leq k+1$
$$ \displaystyle{\prod_{i\neq l} } \|f_{i}\|_{L^{\frac{k \alpha}{\gamma}, \infty}(d \mu)} \|f_{l}\|_{L^{\infty}(d \mu)} \leq C_{k,\alpha, \gamma }
\displaystyle{\sup_{y_{j}}} \ \displaystyle{\prod_{j=1}^{k+1}} \ f_{j}(y_{j}) \det(y_{1}, \dots, y_{k+1})^{\gamma}  \eqno(3.8)$$
holds by symmetry.
\end{lemma}

\begin{proof}
If $\displaystyle{\sup_{y_{j}}} \ \displaystyle{\prod_{j=1}^{k+1}} \ f_{j}(y_{j}) \det(y_{1}, \dots, y_{k+1})^{\gamma} = \infty$, the inequality (3.7) is trivial.
Suppose
$\displaystyle{\sup_{y_{j}}} \ \displaystyle{\prod_{j=1}^{k+1}} \ f_{j}(y_{j}) \det(y_{1}, \dots, y_{k+1})^{\gamma} = A < \infty$,
then there exists measure zero set $E \subset H \times \dots \times H$, $\mu \otimes \cdots \otimes \mu (E)=0$,  such that
$$\displaystyle{\prod_{j=1}^{k+1}} \ f_{j}(y_{j}) \det(y_{1}, \dots, y_{k+1})^{\gamma} \leq A,  \eqno(3.9) $$
 for all $(y_{1}, \dots, y_{k+1}) \in (H \times \dots \times H) \setminus E$.
Note that for any $\varepsilon > 0$, there exists $F \subset H$ such that $\mu (F)>0$, and for all $y_{k+1} \in F$
$$f_{k+1}(y_{k+1}) > \|f_{k+1}\|_{\infty}- \varepsilon. \eqno(3.10)$$
From (3.9) and (3.10) it follows that for $(y_{1}, \dots,y_{k+1}) \in (H \times \dots \times H \times F) \setminus E$,
$$f_{1}(y_{1})  \leq \frac{A}{\|f_{k+1}\|_{\infty}- \varepsilon } \displaystyle{\prod_{j=2}^{k}}  f_{j}(y_{j})^{-1} \det( y_{1}, \dots, y_{k+1})^{-\gamma}.  \eqno(3.11)$$
For any positive $\alpha_{j}$, denote $C_{j}= \{y_{j}: f_{j}(y_{j}) > \alpha_{j}  \}$, $ j=1, \dots,k$.
Note that $\mu \otimes \cdots \otimes \mu (E)=0$, which implies that for almost every $y_{k+1} \in H$,
$$\mu \otimes \cdots \otimes \mu ( \{(y_{1}, \dots, y_{k}):  (y_{1}, \dots, y_{k}, y_{k+1}) \in E  \} ) =0.$$
Denote $\{(y_{1}, \dots, y_{k}):  (y_{1}, \dots, y_{k}, y_{k+1}) \in E  \}$ by $G_{y_{k+1}} \subset H^{k}$.
Since $\mu(F) > 0$,  we can choose a $y_{k+1} \in F$ such that $\mu \otimes \cdots \otimes \mu (G_{y_{k+1}})=0$,
and for all $(y_{1}, \dots, y_{k}) \in H^{k} \setminus G_{y_{k+1}}$
$$(y_{1}, \dots, y_{k}, y_{k+1}) \in (H \times \dots \times H \times F) \setminus E.$$
Since $\mu \otimes \cdots \otimes \mu (G_{y_{k+1}})=0$, for almost every $y_{1} \in H$
 $$\mu \otimes \cdots \otimes \mu  ( \{(y_{2}, \dots, y_{k}) \in H^{k-1} :  (y_{1},y_{2}, \dots, y_{k}) \in G_{y_{k+1}} \} ) =0.$$
That is to say, for  almost every $y_{1}$,  almost every $ (y_{2}, \dots, y_{k}) \in H^{k-1}$
$$ (y_{1},y_{2}, \dots, y_{k}) \in (H \times \dots\times H \times F) \setminus E.$$
Therefore, together with (3.11) implies that for any $\alpha_{1} >0$
\begin{align*}
\ \ \ &\mu (\{y_{1}: f_{1}(y_{1}) > \alpha_{1} \} )  \\
\leq  & \mu ( \{y_{1}: \det(y_{1}, \dots, y_{k+1})^{\gamma}< \frac{A}{\alpha_{1} (\|f_{k+1}\|_{\infty}- \varepsilon)  }  \displaystyle{\prod_{j=2}^{k}}  f_{j}(y_{j})^{-1} ,  \ (y_{2}, \dots, y_{k}) \in H^{k-1} \ a.e. \  \} ).
\end{align*}
Due to the definition of $C_{j}$, we get for any $\alpha_{1} >0 $,
\begin{align*}
&\ \ \ \mu (\{y_{1}: f_{1}(y_{1}) > \alpha_{1} \} )  \\
&\leq \mu ( \{y_{1} \in C_{1} : \det(y_{1}, \dots, y_{k+1})^{\gamma}< \frac{A}{\alpha_{1} (\|f_{k+1}\|_{\infty}- \varepsilon)}
\displaystyle{\prod_{j=2}^{k}} \alpha_{j}^{-1} ,
\ (y_{2}, \dots, y_{k}) \in C_{2} \times \dots \times C_{k} \ a.e. \  \} ) \\
&\leq \mu ( \{y_{1} \in C_{1} : \det(y_{1}, \dots, y_{k+1}) < (\frac{A}{\alpha_{1}
(\|f_{k+1}\|_{\infty}- \varepsilon)})^{\frac{1}{\gamma}}
\displaystyle{\prod_{j=2}^{k}} \alpha_{j}^{-\frac{1}{\gamma}} ,
\ (y_{2}, \dots, y_{k}) \in C_{2} \times \dots \times C_{k} \ a.e. \  \} ) \\
&= \displaystyle{\prod_{j=2}^{k}} \mu(C_{j})^{-1} \mu \otimes \dots \otimes \mu  (\{(y_{1}, \dots,y_{k})\in C_{1} \times \dots \times C_{k}: \det(y_{1}, \dots, y_{k+1}) < (\frac{A}{ \|f_{k+1}\|_{\infty}- \varepsilon })^{\frac{1}{\gamma}}
\displaystyle{\prod_{j=1}^{k}}  \alpha_{j}^{-\frac{1}{\gamma}} \} ).
\end{align*}
Denote $ \frac{A}{ \|f_{k+1}\|_{\infty}- \varepsilon }^{\frac{1}{\gamma}} \displaystyle{\prod_{j=1}^{k}}  \alpha_{j}^{-\frac{1}{\gamma}}  $
by  $M$, then it follows from Lemma 3 that
\begin{align*}
& \ \ \  \mu (\{y_{1}: f_{1}(y_{1})> \alpha_{1}\} ) \\
&\leq \displaystyle{\prod_{j=2}^{k}} \mu(C_{j})^{-1} \mu \otimes \dots \otimes \mu  (\{(y_{1}, \dots,y_{k})\in C_{1} \times \dots \times C_{k}:
\det(y_{1}, \dots, y_{k+1}) < M \} ) \\
&=\displaystyle{\prod_{j=2}^{k}} \mu(C_{j})^{-1}  \mu \otimes \dots \otimes \mu  (\{(y_{1}, \dots,y_{k})\in C_{1} \times \dots \times C_{k}: \det(0,y_{1}-y_{k+1}, \dots, y_{k}-y_{k+1}) <M \}) \\
&\leq C_{k, \alpha} \displaystyle{\prod_{j=2}^{k}} \mu(C_{j})^{-1} M^{\alpha}
\displaystyle{\prod_{j=1}^{k}} \mu(C_{j})^{1- \frac{1}{k}} \\
&= C_{k, \alpha} M^{\alpha}
\mu(C_{1})^{1-\frac{1}{k}} \mu(C_{2})^{-\frac{1}{k}} \dots \mu(C_{k})^{-\frac{1}{k}}.
\end{align*}
Hence
$$\mu (\{y_{1}: f_{1}(y_{1}) > \alpha_{1}\} ) = \mu(C_{1})
\leq C_{k, \alpha} (\frac{A}{ \|f_{k+1}\|_{\infty}- \varepsilon })^{\frac{\alpha}{\gamma}} \prod_{j=1}^{k} \alpha_{j}^{-\frac{\alpha}{\gamma}}
\mu(C_{1})^{1-\frac{1}{k}} \mu(C_{2})^{-\frac{1}{k}} \dots \mu(C_{k})^{-\frac{1}{k}}.$$
That is, we obtain for any $\alpha_{j} > 0$
$$\mu(C_{1})^{\frac{\gamma}{\alpha}}
\leq C_{k, \alpha}^{\frac{\gamma}{\alpha}} \frac{A}{ \|f_{k+1}\|_{\infty}- \varepsilon }
\displaystyle{\prod_{j=1}^{k}} \alpha_{j}^{-1} \mu(C_{1})^{(1-\frac{1}{k})\frac{\gamma}{\alpha}} \mu(C_{2})^{-\frac{\gamma}{k \alpha}} \dots \mu(C_{k})^{-\frac{\gamma}{k \alpha}}.$$
Simplify it to give that
$$\alpha_{1} \mu(C_{1})^{\frac{\gamma}{k \alpha}}
\leq C_{k, \alpha}^{\frac{\gamma}{\alpha}} \frac{A}{ \|f_{k+1}\|_{\infty}- \varepsilon }
\frac{1}{ \displaystyle{\prod_{j=2}^{k}} \alpha_{j} \mu(C_{j})^{\frac{\gamma}{k \alpha}} } . $$
Let $\varepsilon \rightarrow 0$, we get for any $\alpha_{j} > 0$
$$\displaystyle{\prod_{j=1}^{k}} \alpha_{j} \mu(C_{j})^{\frac{\gamma}{k \alpha}}
\leq C_{k, \alpha}^{\frac{\gamma}{\alpha}}
\frac{A}{ \|f_{k+1}\|_{\infty} } .   \eqno(3.12) $$
Since $\alpha_{j}$  are arbitrary, this  allows us to take
the infimum over all $\alpha_{j} > 0$ on (3.12), $j=1, \dots, k+1$,  which gives
$$\displaystyle{\prod_{j=1}^{k}} \|f_{j}\|_{L^{\frac{k \alpha}{\gamma}, \infty}(d \mu)} \|f_{k+1}\|_{L^{\infty}(d \mu)} \leq C_{k, \alpha}^{\frac{\gamma}{\alpha}} A
=C_{k, \alpha}^{\frac{\gamma}{\alpha}} \displaystyle{\sup_{y_{j}}} \ \displaystyle{\prod_{j=1}^{k+1}} \ f_{j}(y_{j}) \det(y_{1}, \dots,y_{k+1})^{\gamma}.$$
This proves the endpoint case (3.7).  Meanwhile by symmetry (3.8) holds.

\end{proof}

\medskip

\hspace{-13pt}{\it Proof of Theorem 3.1 (a)}\quad

For any general $f_{j} \in L^{p_{j}}(d \mu)$, there exist sequences of simple functions $\{f_{jn}\} \uparrow f_{j}$ as $n \rightarrow \infty$. We apply Lemma 3.2
for simple functions  $f_{jn}$ , this is because simple functions are in $L^{\frac{k \alpha}{\gamma}, \infty}(d \mu) \cap L^{\infty}(d \mu)$.
For each $1 \leq j \leq k+1$, for every $n$, we have
$$ \prod_{i \neq j}  \|f_{in}\|_{L^{\frac{k \alpha}{\gamma}, \infty}(d \mu)}  \|f_{jn}\|_{L^{\infty}(d \mu)}  \lesssim
\displaystyle{\sup_{y_{j}}} \ \displaystyle{\prod_{j=1}^{k+1}} \ f_{jn}(y_{j}) \det(y_{1}, \dots,y_{k+1})^{\gamma}. \eqno(3.13)$$
Based on this, by the layer cake representation it is easy to obtain that for $\frac{1}{p_{j}} < \frac{\gamma}{k \alpha}$, $f_{jn} \in L^{p_{j}}(d \mu)$ and
$$\|f_{jn}\|_{L^{p_{j}}(d \mu)} \lesssim \|f_{jn}\|_{L^{\frac{k \alpha}{\gamma}, \infty}(d \mu)}^{\frac{k \alpha}{\gamma p_{j}}}
\|f_{jn}\|_{L^{\infty}(d \mu)}^{1-\frac{k \alpha}{\gamma p_{j}}}.   \eqno(3.14)$$
We assume that for every $n$
$$\displaystyle{\sup_{y_{j}}} \ \displaystyle{\prod_{j=1}^{k+1}} \ f_{jn}(y_{j}) \det(y_{1}, \dots,y_{k+1})^{\gamma}=A_{n}<\infty,$$
then from (3.13), (3.14) and  $\displaystyle{ \sum_{j=1}^{k+1} } \ \frac{1}{p_{j}}= \frac{\gamma}{\alpha}$ it follows that
\begin{align*}
\displaystyle{\prod_{j=1}^{k+1}} \|f_{jn}\|_{L^{p_{j}}(d \mu)}
&\lesssim \prod_{j=1}^{k+1} \|f_{jn}\|_{L^{\frac{k \alpha}{\gamma}, \infty}(d \mu)}^{\frac{k \alpha}{\gamma p_{j}}}  \|f_{jn}\|_{L^{\infty}(d \mu)}^{1-\frac{k \alpha}{\gamma p_{j}}} \\
&=\prod_{j=1}^{k+1} ( \prod_{i \neq j}  \|f_{in}\|_{L^{\frac{k \alpha}{\gamma}, \infty}(d \mu)}  \|f_{jn}\|_{L^{\infty}(d \mu)} )^{1- \frac{k \alpha}{\gamma p_{j} }} \\
&\lesssim \prod_{j=1}^{k+1} (\displaystyle{\sup_{y_{j}}} \ \displaystyle{\prod_{j=1}^{k+1}} \ f_{jn}(y_{j}) \det(y_{1}, \dots,y_{k+1})^{\gamma})^{1- \frac{k \alpha}{\gamma p_{j}}} \\
&= \prod_{j=1}^{k+1} A_{n}^{1- \frac{k \alpha}{\gamma p_{j} }}.
\end{align*}
Note that  $\sum\limits_{j=1}^{k+1} (1- \frac{k \alpha}{\gamma p_{j}})=1$, since $\displaystyle{ \sum_{j=1}^{k+1} } \ \frac{1}{p_{j}}= \frac{\gamma}{\alpha}$.
Hence,
$$\displaystyle{\prod_{j=1}^{k+1}} \|f_{jn}\|_{L^{p_{j}}(d \mu)} \lesssim   \prod_{j=1}^{k+1} A_{n}^{1- \frac{k \alpha}{\gamma p_{j} }}
= A_{n} \equiv \displaystyle{\sup_{y_{j}}} \ \displaystyle{\prod_{j=1}^{k+1}} \ f_{jn}(y_{j}) \det(y_{1}, \dots,y_{k+1})^{\gamma}.$$
Therefore, for every $n$
$$\displaystyle{\prod_{j=1}^{k+1}} \|f_{jn}\|_{L^{p_{j}}(d \mu)}
\leq \displaystyle{\sup_{y_{j}}} \ \displaystyle{\prod_{j=1}^{k+1}} \ f_{jn}(y_{j}) \det(y_{1}, \dots,y_{k+1})^{\gamma}
\leq \displaystyle{\sup_{y_{j}}} \ \displaystyle{\prod_{j=1}^{k+1}} \ f_{j}(y_{j}) \det(y_{1}, \dots,y_{k+1})^{\gamma}.$$
Let $n \rightarrow \infty$ to deduce that
$$\displaystyle{\prod_{j=1}^{k+1}} \|f_{j}\|_{L^{p_{j}}(d \mu)}  \lesssim \displaystyle{\sup_{y_{j}}} \ \displaystyle{\prod_{j=1}^{k+1}} \ f_{j}(y_{j}) \det(y_{1}, \dots,y_{k+1})^{\gamma}.$$
This completes the proof of this theorem.

$\hfill\Box$

\bigskip

We shall now present an alternative  method  to show that condition (i) implies condition (ii),  mainly applying  Gressman's  result Lemma 5 above. \\

\hspace{-13pt}{\it Alternative proof of Theorem 3.1 (a)}\quad

Suppose $\displaystyle{\sup_{y_{j}}} \ \displaystyle{\prod_{j=1}^{k+1}} \ f_{j}(y_{j})  \det(y_{1}, \dots , y_{n+1})^{\gamma} = A < \infty$.
We can write
\begin{align*}
& \ \ \ \|f_{1}\|_{p_{1}}^{p_{1}} \cdots\|f_{k+1}\|_{p_{k+1}}^{p_{k+1}} \\
&= \int_{H} \cdots \int_{H} \displaystyle{\prod_{j=1}^{k+1}} \ f_{j}(y_{j})^{p_{j}} d\mu(y_{1}) \dots d\mu(y_{k+1}) \\
&= \int_{H} \cdots \int_{H} \displaystyle{\prod_{j=1}^{k+1}}  \ f_{j}(y_{j}) \det(y_{1}, \dots,y_{k+1})^{\gamma}
\displaystyle{\prod_{j=1}^{k+1}}  \ f_{j}(y_{j})^{p_{j}-1} \det(y_{1}, \dots,y_{k+1})^{-\gamma}
d\mu(y_{1}) \dots d\mu(y_{k+1})
\end{align*}
Since
$$ \displaystyle{\sup_{y_{j}}} \ \displaystyle{\prod_{j=1}^{k+1}} \  f_{j}(y_{j}) \det(y_{1}, \dots,y_{k+1})^{\gamma} \leq A,$$
and $p_{j}-1=\frac{p_{j}}{p_{j}^{\prime}}$, so
$$\|f_{1}\|_{p_{1}}^{p_{1}} \cdots \|f_{k+1}\|_{p_{k+1}}^{p_{k+1}}
\leq A \int_{H} \cdots \int_{H}  \displaystyle{\prod_{j=1}^{k+1}} \  f_{j}(y_{j})^{\frac{p_{j}}{p_{j}^{\prime}}}
\det(y_{1}, \dots,y_{k+1})^{-\gamma} d\mu(y_{1}) \dots d\mu(y_{k+1}). $$
From Lemma 5 it follows that for all $1 \leq p_{j}^{\prime} \leq \infty$ satisfying $\frac{1}{p_{j}^{\prime}} > 1- \frac{\gamma}{k \alpha}$, $j=1, \dots, k+1$,
$$\int_{H}\cdots \int_{H}  \displaystyle{\prod_{j=1}^{k+1}} \ f_{j}(y_{j})^{\frac{p_{j}}{p_{j}^{\prime}}} \det(y_{1}, \dots,y_{k+1})^{-\gamma} d\mu(y_{1}) \dots d\mu(y_{k+1})
\leq C \displaystyle{\prod_{j=1}^{k+1}} \|f_{j}^{\frac{p_{j}}{p_{j}^{\prime}}}\|_{p_{j}^{\prime}}$$
holds, where $k+1 - \displaystyle{ \sum_{j=1}^{k+1} } \ \frac{1}{p_{j}^{\prime}}= \displaystyle{ \sum_{j=1}^{k+1} } \ \frac{1}{p_{j}}= \frac{\gamma}{\alpha}$. \\
Therefore,
$$\|f_{1}\|_{p_{1}}^{p_{1}} \cdots \|f_{k+1}\|_{p_{k+1}}^{p_{k+1}}
\leq C A \ \displaystyle{\prod_{j=1}^{k+1}} \|f_{j}^{\frac{p_{j}}{p_{j}^{\prime}}}\|_{p_{j}^{\prime}}
= C A \ \displaystyle{\prod_{j=1}^{k+1}} \|f_{j}\|_{p_{j}}^{p_{j}-1},$$
which indicates for all $1 \leq p_{j} < \infty$ satisfying $\frac{1}{p_{j}} < \frac{\gamma}{k \alpha}$ and $\displaystyle{ \sum_{j=1}^{k+1} } \ \frac{1}{p_{j}}= \frac{\gamma}{\alpha}$,
we have
$$\displaystyle{\prod_{j=1}^{k+1}} \|f_{j}\|_{p_{j}} \leq C \
\displaystyle{\sup_{y_{j}}} \ \displaystyle{\prod_{j=1}^{k+1}} \ f_{j}(y_{j}) \det(y_{1}, \dots,y_{k+1})^{\gamma}.  \eqno(3.15) $$

As for other $0 < p_{j} < 1$, $1 \leq j \leq k+1$, such that
$\frac{1}{p_{j}} < \frac{\gamma}{k \alpha}$ and $\displaystyle{ \sum_{j=1}^{k+1} } \ \frac{1}{p_{j}}= \frac{\gamma}{\alpha}$, it is easy to see
$\displaystyle{\prod_{j=1}^{k+1}} \|f_{j}\|_{p_{j}}
= \displaystyle{\prod_{j=1}^{k+1}} \|f_{j}^{p_{1} \dots p_{k+1}}\|_{\frac{1}{q_{j}}}^{ \frac{1}{p_{1} \dots p_{k+1}}}$,
with $q_{j}= p_{1} \dots p_{j-1} p_{j+1} \dots p_{k+1}$, and
$\displaystyle{ \sum_{j=1}^{k+1} } q_{j}= \frac{\gamma}{\alpha}(p_{1} \dots p_{k+1})$.
Since $\frac{1}{q_{j}} >1$, $q_{j} < \frac{\gamma}{k \alpha}$, we can apply (3.15) to give
$$\displaystyle{\prod_{j=1}^{k+1}} \|f_{j}^{p_{1} \dots p_{k+1}}\|_{\frac{1}{q_{j}}}
\leq C \displaystyle{\sup_{y_{j}}} \ \displaystyle{\prod_{j=1}^{k+1}} \ f_{j}(y_{j})^{p_{1} \dots p_{k+1}} \det(y_{1}, \dots,y_{k+1})^{\gamma (p_{1} \dots p_{k+1})}.$$
Thus
\begin{align*}
 \displaystyle{\prod_{j=1}^{k+1}} \|f_{j}\|_{p_{j}}
&= \displaystyle{\prod_{j=1}^{k+1}} \|f_{j}^{p_{1} \dots p_{k+1}}\|_{\frac{1}{q_{j}}}^
{\frac{1}{p_{1} \dots p_{k+1}}} \\
&\leq C \ \displaystyle{\sup_{y_{j}}} \ \displaystyle{\prod_{j=1}^{k+1}} \ f_{j}(y_{j}) \det(y_{1}, \dots, y_{k+1})^{\gamma}.
\end{align*}
In conclusion, we obtain that for all $0 < p_{j} < \infty$ satisfying
$\frac{1}{p_{j}} < \frac{\gamma}{k \alpha}$ and
$\displaystyle{ \sum_{j=1}^{k+1} } \ \frac{1}{p_{j}}= \frac{\gamma}{\alpha}$,  $1 \leq j \leq k+1$,
$$ \displaystyle{\prod_{j=1}^{k+1}} \|f_{j}\|_{p_{j}} \leq C \ \displaystyle{\sup_{y_{j}}} \ \displaystyle{\prod_{j=1}^{k+1}} \ f_{j}(y_{j}) \det(y_{1}, \dots,y_{k+1})^{\gamma}. $$
This also completes the proof of part  (a) of Theorem 3.1.

$\hfill\Box$

\medskip

It should be pointed out that we find condition (i) and (ii) are equivalent to the  inequality (3.3) in Lemma 5 as well from the the alternative method of proof (a).
Lemma 5 states condition (i) implies inequality (3.3), and we use inequality (3.3) to get the inequality (3.5) in the alternative method of proof (a).
Besides, Theorem 3.1 shows that condition  (i) and (ii)  i.e.  inequality (3.5) are equivalent.

\bigskip

If we strengthen  the condition (i)  to  $\mu(B) \sim  |B|_{k}^{\alpha}$ for all ellipsoids $B$ in $H$,
then  $ \frac{1}{p_{j}} < \frac{\gamma}{k \alpha}$  for all $1 \leq j \leq k+1$ and
$\frac{\gamma}{\alpha} = \displaystyle{ \sum_{j=1}^{k+1} } \ \frac{1}{p_{j}} $ are  necessary and sufficient
conditions for inequality (3.5) to hold, which can be seen in the following theorem.

\medskip

\begin{theorem}
Let $f_{j}$ be nonnegative measurable functions defined on real  finite-dimensional Hilbert space $H$. Let
$\mu$ be a $\sigma$-finite, nonnegative Borel measure with satisfying
$\mu(B) \sim  |B|_{k}^{\alpha}$ for all ellipsoids $B$ in $H$. Then for all $0< p_{j}<\infty$
$$ \displaystyle{\prod_{j=1}^{k+1}} \|f_{j}\|_{L^{p_{j}}(d \mu)} \leq
C_{k, \alpha, p_{j}} \displaystyle{\sup_{y_{j}}} \ \displaystyle{\prod_{j=1}^{k+1}} \ f_{j}(y_{j}) \ (\det(y_{1}, \dots,y_{k+1}))^{\gamma}  \eqno(3.16)$$
holds,
if  and only if $p_{j}$ satisfy
\begin{center}
 $ \frac{1}{p_{j}} < \frac{\gamma}{k \alpha}$  \  for all $1 \leq j \leq k+1$ \  and  \ $\frac{\gamma}{\alpha} = \displaystyle{ \sum_{j=1}^{k+1} } \ \frac{1}{p_{j}} $ .
\end{center}
\end{theorem}

\begin{proof}
$\mu(B) \sim  |B|_{k}^{\alpha}$ for all ellipsoids $B$ in $H$, so the measure $\mu$
is $k$-curved with exponent $\alpha$. Theorem 3.1 (a) gave the sufficient conditions for
inequality (3.16) to hold.
To see the converse,  we study the necessary conditions for inequality (3.16) to hold.
Suppose (3.16) holds for all nonnegative functions $f_{j} \in L^{p_{j}}(d \mu)$, then
$\frac{\gamma}{\alpha}  = \displaystyle{ \sum_{j=1}^{k+1} } \ \frac{1}{p_{j}}$
which follows from  homogeneity. \\
Let $f_{j}= \chi_{B}$ where $B$ is a ball in $H$,  $j=1, \dots,k+1$.
We consider functions $\chi_{B}(\frac{\cdot}{R})$ for all $R >0$: \  for $j=1, \dots,k+1$, we have
$$\|\chi_{B}(\frac{\cdot}{R})\|_{L^{p_{j}}(d \mu)} \sim R^{\frac{k \alpha}{p_{j}}}( \mu(B))^{\frac{1}{p_{j}}},$$
this is because for all $R >0$
$$ \mu(R B) \sim  |R B|_{k}^{\alpha} = R^{k \alpha} |B|_{k}^{\alpha} \sim R^{k \alpha} \mu(B).$$
From the property of $\det(y_{1}, \dots,y_{k+1})$ it follows that
\begin{align*}
&\ \ \ \ \ \ \displaystyle{\sup_{y_{j}}} \ \displaystyle{\prod_{j=1}^{k+1}} \ \chi_{B}(\frac{y_{j}}{R}) \det(y_{1}, \dots,y_{k+1})^{\gamma} \\
&=  R^{k \gamma} \displaystyle{\sup_{y_{j}}}  \ \displaystyle{\prod_{j=1}^{k+1}} \ \chi_{B}(\frac{y_{j}}{R})  \det(\frac{y_{1}}{R}, \dots,\frac{y_{k+1}}{R})^{\gamma} \\
&=  R^{k \gamma} \displaystyle{\sup_{y_{j}}} \ \displaystyle{\prod_{j=1}^{k+1}} \ \chi_{B}(y_{j}) \det(y_{1}, \dots,y_{k+1})^{\gamma}.
\end{align*}
So if (3.16) holds, then
$\displaystyle{ \prod_{j=1}^{k+1} } \ R^{\frac{k \alpha}{p_{j}}} \lesssim  R^{k \gamma}$ for all $R >0$,
which implies
$$\displaystyle{ \sum_{j=1}^{k+1} } \ \frac{k \alpha}{p_{j}} = k \gamma .$$
That is
$$\frac{\gamma}{\alpha} = \displaystyle{ \sum_{j=1}^{k+1} } \ \frac{1}{p_{j}}.    \eqno(3.17)$$

We now claim that if (3.16) holds for all nonnegative functions $f_{j} \in L^{p_{j}}(d \mu)$,
$p_{j}$ must satisfy $\frac{1}{p_{j}} < \frac{\gamma}{k \alpha}$  for all $1 \leq j \leq k+1$.
Let $f_{1} \in L^{p_{1}}(d \mu)$ be supported on $\{y_{1}: |y_{1}| \geq 10  \}$.
For $2 \leq j \leq k+1$, let $f_{j}= \chi_{B(0, \frac{1}{2})}$ where
$B(0, \frac{1}{2})$ denotes the ball in $H$ centred at $0$ with radius $\frac{1}{2}$.
So $|y_{1}-y_{j}| \sim |y_{1}|$ for all $2 \leq j \leq k+1$.
We consider the new functions $f_{1}, f_{j}(\frac{\cdot}{\epsilon})$
with $0< \epsilon < 1$,  $2 \leq j \leq k+1$.

Suppose that inequality $(3.16)$ holds for all nonnegative functions $f_{j} \in L^{p_{j}}(d \mu)$, then
$$\|f_{1}\|_{L^{p_{1}}(d \mu)}
\displaystyle{\prod_{j=2}^{k+1}}  \|f_{j}(\frac{\cdot}{\epsilon})\|_{L^{p_{j}}(d \mu)}
\lesssim  \ \sup\limits_{y_{j}} \ f_{1}(y_{1})    \prod_{j=2}^{k+1}  f_{j}(\frac{y_{j}}{\epsilon})  \det(y_{1}, \dots,y_{k+1})^{\gamma} .$$
By the  Hadamard inequality
$$\det(y_{1}, \dots,y_{k+1}) \leq |y_{1}-y_{k+1}|  |y_{2}-y_{k+1}| \cdots \ |y_{k}-y_{k+1}|,$$
we  have
\begin{align*}
&\ \ \ \  \sup\limits_{y_{j}} \ f_{1}(y_{1})    \prod_{j=2}^{k+1}  f_{j}(\frac{y_{j}}{\epsilon})  \det(y_{1}, \dots,y_{k+1})^{\gamma}  \\
& \leq \ \sup\limits_{y_{j}} \  f_{1}(y_{1})   \prod_{j=2}^{k+1}  f_{j}(\frac{y_{j}}{\epsilon})  \  (|y_{1}-y_{k+1}|  |y_{2}-y_{k+1}| \cdots \ |y_{k}-y_{k+1}|)^{\gamma}  \\
& \sim  \epsilon^{(k-1) \gamma}  \ \sup\limits_{y_{j}} \ f_{1}(y_{1})  \prod_{j=2}^{k+1}  f_{j}(\frac{y_{j}}{\epsilon})  \
(|y_{1}-\frac{y_{k+1}}{\epsilon}| |\frac{y_{2}}{\epsilon}-\frac{y_{k+1}}{\epsilon}|  \cdots \ |\frac{y_{k}}{\epsilon}-\frac{y_{k+1}}{\epsilon}|)^{\gamma} \\
&\sim \epsilon^{(k-1) \gamma}  \ \sup\limits_{y_{j}} \ \displaystyle{\prod_{j=1}^{k+1}} \ f_{j}(y_{j})  \  (|y_{1}-y_{k+1}|  |y_{2}-y_{k+1}| \cdots \ |y_{k}-y_{k+1}|)^{\gamma}.
\end{align*}
On the other hand, for $2 \leq j \leq k+1$
$$\|\chi_{B(0, \frac{1}{2})}(\frac{\cdot}{\epsilon})\|_{L^{p_{j}}(d \mu)} \sim
 \epsilon^{ \frac{k \alpha}{p_{j}}} \mu(B(0, \frac{1}{2}))^{\frac{1}{p_{j}}},$$
this is because for  $ \epsilon >0$
$$ \mu(\epsilon B(0, \frac{1}{2}))
\sim  | \epsilon B(0, \frac{1}{2})|_{k}^{\alpha}
= \epsilon^{k \alpha} |B(0, \frac{1}{2})|_{k}^{\alpha}
\sim \epsilon^{k \alpha} \mu(B(0, \frac{1}{2})).$$
Then
$$\|f_{1}\|_{L^{p_{1}}(d \mu)}
\displaystyle{\prod_{j=2}^{k+1}}  \|f_{j}(\frac{\cdot}{\epsilon})\|_{L^{p_{j}}(d \mu)}
= \prod_{j=2}^{k+1}  \epsilon^{ \frac{k \alpha}{p_{j}} } \
\displaystyle{\prod_{j=1}^{k+1}} \|f_{j}\|_{L^{p_{j}}(d \mu)}.$$
So if (3.16) holds, then for all $0< \epsilon < 1$,
$$\displaystyle{\prod_{j=2}^{k+1}}  \epsilon^{ \frac{k \alpha}{p_{j}} }  \lesssim \epsilon^{(k-1) \gamma},$$
then we have
$$\sum_{j=2}^{k+1} \frac{k \alpha}{p_{j}}  \geq (k-1)  \gamma,$$
which means
$$\frac{1}{p_{1}}= \frac{\gamma}{\alpha}- \sum_{j=2}^{k+1} \frac{1}{p_{j}}
\leq \frac{\gamma}{\alpha}- \frac{k-1}{k \alpha} \gamma
= \frac{\gamma}{k \alpha}.   \eqno(3.18)$$
By symmetry, for any $1 \leq j \leq k+1$
we have $\frac{1}{p_{j}} \leq \frac{\gamma}{k \alpha}$
provided (3.16) holds for all nonnegative functions $f_{j} \in L^{p_{j}}(d \mu)$. \\

As for the boundary case, the following counterexample shows that we must have
$\frac{1}{p_{j}} < \frac{\gamma}{k \alpha}$ for all $1 \leq j \leq k+1$ . \\
For any positive $N$, let $f_{1}(y_{1})= \frac{1}{|y_{1}|^{\gamma}} \chi_{2 \leq |y_{1}| \leq N }$,
$f_{j}(y_{j})= \chi_{|y_{j}| \leq 1/4}, 2 \leq j \leq k+1$.
The  Hadamard inequality tells us
$$\det(y_{1},  \dots ,y_{k+1}) \leq |y_{1}-y_{k+1}|  \cdots |y_{k}-y_{k+1}|,$$
then
\begin{align*}
\displaystyle{\sup_{y_{j}}} \ \displaystyle{\prod_{j=1}^{k+1}} \ f_{j}(y_{j}) \det(y_{1}, \dots,y_{k+1})^{\gamma}
&\leq \displaystyle{\sup_{ y_{j}}  }  \ \displaystyle{\prod_{j=1}^{k+1}} \ f_{j}(y_{j}) \  |y_{1}-y_{k+1}|^{\gamma} \cdots |y_{k}-y_{k+1}|^{\gamma}  \\
&\lesssim \displaystyle{\sup_{2 \leq |y_{1}| \leq N}} |y_{1}|^{-\gamma} (|y_{1}|+\frac{1}{4})^{\gamma}
\lesssim 1.
\end{align*}
On the other hand, by polar coordinates we obtain
\begin{align*}
\displaystyle{ \limsup_{N \rightarrow \infty }} \ \|f_{1}\|_{L^{\frac{k \alpha}{\gamma}(d \mu)}}
&= \limsup_{N \rightarrow \infty  } \int_{2 \leq |y_{1}| \leq N} |y_{1}|^{-k \alpha} d \mu(y_{1}) \\
&= \sum\limits_{j>0} \int_{|y_{1}| \sim 2^{j}} |y_{1}|^{-k \alpha} d \mu(y_{1}) \\
&\gtrsim \sum\limits_{j>0} 2^{-k \alpha j} 2^{j k \alpha } =\infty,
\end{align*}
which gives the contradiction to (3.16). \\
The last inequality follows due to  the fact that
$\mu(B) \sim  |B|_{k}^{\alpha}$ for all ellipsoids $B$ in $H$, which implies
$$\mu(\{y_{1} \in H: |y_{1}| \sim 2^{j} \}) \sim 2^{j k \alpha}.$$

\end{proof}

\bigskip

As is well known, the Lebesgue measure on $\mathbb{R}^{n}$ is not only $n$-curved with exponent $1$,
but also it satisfies $|B| \sim |B|_{n}$ for all ellipsoids $B$ in $H$.
Hence we obtain the following corollary immediately.

\begin{corollary}
Let $ f_{j} \in  L^{p_{j}}(\mathbb{R}^{n})  $ with Lebesgue measure, then
$$  \displaystyle{\prod_{j=1}^{n+1}} \|f_{j}\|_{p_{j}}   \leq
C_{n, p_{j}} \displaystyle{\sup_{y_{j}}} \ \displaystyle{\prod_{j=1}^{n+1}} \ f_{j}(y_{j}) \det(y_{1}, \dots, y_{n+1})^{\gamma}   \eqno(3.19)$$
holds, if  and only if $p_{j}$ satisfy
\begin{center}
 $ \frac{1}{p_{j}} < \frac{\gamma}{n}$  \ for all $1 \leq j \leq n+1$ and $\gamma = \displaystyle{ \sum_{j=1}^{n+1} } \ \frac{1}{p_{j}} $.
\end{center}

\end{corollary}

\bigskip

We now consider the second class of multilinear inequalities where we have a product form  rather than a determinant.

\medskip

\begin{theorem}
Let $ r_{ij} >0 $  and  $r_{i j} = r_{j i}$.
Let $f_{j}$  be nonnegative measurable functions defined on $\mathbb{R}^{n}$, then
$$\prod_{j=1}^{3} \|f_{j}\|_{p_{j}} \leq C_{p_{j},  r_{ij}, n }  \ \displaystyle{\sup_{y_{j}}}  \prod_{j=1}^{3}  f_{j}(y_{j})  \prod_{1 \leq i < j \leq 3} |y_{i}-y_{j}|^{r_{ij}} \eqno(3.20)$$
holds,  if and only if   $p_{j}$ satisfy
\begin{center}
 $\displaystyle{ \sum_{j=1}^{3}   }  \frac{1}{p_{j}}= \frac{1}{n} ( r_{12} +r_{13}+ r_{23})$ ,
$\frac{1}{p_{j}} < \frac{1}{n}    \displaystyle{\sum_ {i \neq j}}   r_{ij} $  \  for every $j$.
\end{center}
\end{theorem}

\begin{proof}
$\displaystyle{ \sum_{j=1}^{3}   }  \frac{1}{p_{j}}= \frac{1}{n} ( r_{12} +r_{13}+ r_{23})$
just follows from homogeneity. Besides,
by applying the similar example in the proof of Theorem 3.3  we can get the necessary conditions for (3.20) to hold: for every $ j$
$$\frac{1}{p_{j}} \leq \frac{1}{n}    \displaystyle{\sum_ {i \neq j}}   r_{ij}. $$
The following counterexample shows that we must have $\frac{1}{p_{j}} < \frac{1}{n}    \displaystyle{\sum_ {i \neq j}}   r_{ij} $ for each $j$.
If we assume $\frac{1}{p_{1}}= \frac{r_{12}+r_{13}}{n} $,
for any positive $N$, let
\begin{center}
$f_{1}(y_{1})= |y_{1}|^{-(r_{12}+r_{13})} \chi_{2 \leq |y_{1}| \leq N }$, \
$f_{2}(y_{2})= \chi_{|y_{2}| \leq 1/4}$,  \  $f_{3}(y_{3})= \chi_{|y_{3}| \leq 1/4}$.
\end{center}
Suppose
$$A=  \displaystyle{\sup_{y_{j}}} \  f_{1}(y_{1})f_{2}(y_{2})f_{3}(y_{3}) |y_{1}-y_{2}|^{r_{12}} |y_{1}-y_{3}|^{r_{13}} |y_{2}-y_{3}|^{r_{23}},$$
then
\begin{align*}
A &\lesssim \displaystyle{\sup_{2 \leq |y_{1}| \leq N}} |y_{1}|^{-(r_{12}+r_{13})} (|y_{1}|+\frac{1}{4})^{r_{12}} (|y_{1}|+\frac{1}{4})^{r_{13}} \\
&\lesssim 1.
\end{align*}
However, by polar coordinates we obtain
\begin{align*}
\|f_{1}\|_{\frac{n}{r_{12}+r_{13}}}
&=\int_{2 \leq |y_{1}| \leq N} \frac{1}{|y_{1}|^{n}} dy_{1}\\
&=C \int_{2}^{N} \frac{r^{n-1}}{r^{n}} dr\\
&= C(\ln N-\ln 2)  \rightarrow \infty,
\end{align*}
as $N \rightarrow \infty$.

\bigskip

For the converse, suppose
$$A=  \displaystyle{\sup_{y_{j}}} \  f_{1}(y_{1})f_{2}(y_{2})f_{3}(y_{3}) |y_{1}-y_{2}|^{r_{12}} |y_{1}-y_{3}|^{r_{13}} |y_{2}-y_{3}|^{r_{23}} < \infty,$$
then there exists measure zero set $E \subset \mathbb{R}^{n} \times  \mathbb{R}^{n} \times  \mathbb{R}^{n}$, such that
$$f_{1}(y_{1})f_{2}(y_{2})f_{3}(y_{3}) |y_{1}-y_{2}|^{r_{12}} |y_{1}-y_{3}|^{r_{13}} |y_{2}-y_{3}|^{r_{23}}
\leq A,$$
for all $(y_{1}, y_{2}, y_{3}) \in (\mathbb{R}^{n} \times \mathbb{R}^{n} \times  \mathbb{R}^{n}) \setminus E$. \\
By the definition of $\|f_{3}\|_{\infty}$,   for any $\varepsilon > 0$  there exists $F \subset \mathbb{R}^{n} $ such that $|F|>0$, and for all $y_{3} \in F$
 $$f_{3}(y_{3})  > \|f_{3}\|_{\infty}- \varepsilon.$$
So for all $(y_{1},...,y_{3}) \in (\mathbb{R}^{n} \times \mathbb{R}^{n} \times F) \setminus E$,
$$f_{2}(y_{2}) (\|f_{3}\|_{\infty}- \varepsilon) \leq \frac{1}{|y_{1}-y_{2}|^{r_{12}} |y_{2}-y_{3}|^{r_{23}}} \frac{A}{|y_{1}-y_{3}|^{r_{13}} f_{1}(y_{1})}.$$
Since $|E|=0$, for almost every $y_{3} \in \mathbb{R}^{n}$
$$| \{(y_{1},y_{2}) \in \mathbb{R}^{n} \times \mathbb{R}^{n}: (y_{1},y_{2},y_{3}) \in E \}|=0.$$
Denote $ \{(y_{1},y_{2}) \in \mathbb{R}^{n} \times \mathbb{R}^{n} : (y_{1},y_{2},y_{3}) \in E \}$ by $G_{y_{3}}$.
Because  $|F|>0$,  we can choose a $y_{3} \in F$ such that $|G_{y_{3}}|=0 $,
which  implies  for almost every $y_{1} \in \mathbb{R}^{n}$,
$$| \{ y_{2} \in \mathbb{R}^{n} : (y_{1},y_{2}) \in G_{y_{3}} \}|=0.$$
That means for almost every $ y_{1}$,  almost every $y_{2}$
$$ (y_{1},y_{2},y_{3}) \in  (\mathbb{R}^{n} \times \mathbb{R}^{n} \times \mathbb {R}^{n}) \setminus E.$$
Thus for almost every $y_{1}$, any small $\theta >0$,
\begin{align*}
\|f_{2}\|_{\frac{n}{r_{12}+r_{23}-\theta}} ( \|f_{3}\|_{\infty}- \varepsilon )
&\leq ( \int_{\mathbb{R}^{n}}  (|y_{1}-y_{2}|^{-r_{12}} |y_{2}-y_{3}|^{-r_{23}})^{\frac{n}{r_{12}+r_{23}-\theta}} dy_{2} )^{\frac{r_{12}+r_{23}-\theta}{n}}  \frac{A}{|y_{1}-y_{3}|^{r_{13}} f_{1}(y_{1})} \\
&= C (|y_{1}-y_{3}|^{n- \frac{r_{12}n + r_{23}n}{r_{12}+r_{23}-\theta}} )^{\frac{r_{12}+r_{23}-\theta}{n}}  \frac{A}{|y_{1}-y_{3}|^{r_{13}} f_{1}(y_{1})} \\
&= C |y_{1}-y_{3}|^{-\theta} \frac{A}{|y_{1}-y_{3}|^{r_{13}} f_{1}(y_{1})} \\
&= C \frac{A}{|y_{1}-y_{3}|^{r_{13}+\theta} f_{1}(y_{1})}.
\end{align*}
Take the infimum over $y_{1}$, then let $\varepsilon \rightarrow 0$,
$$\|f_{2}\|_{\frac{n}{r_{12}+r_{23}-\theta}} \|f_{3}\|_{\infty}   \leq C \   \displaystyle{\inf_{y_{1}} } \frac{A}{|y_{1}-y_{3}|^{r_{13}+\theta} f_{1}(y_{1})}
= C \  \frac{A}{\displaystyle{\sup_{y_{1}} } \ |y_{1}-y_{3}|^{ r_{13}+\theta} f_{1}(y_{1})}.  \eqno(3.21)$$
In  the proof of Lemma 2.2, we have stated that for the bilinear form,
$$\|f_{1}\|_{\frac{n}{r_{13}+\theta}, \infty}  \lesssim  \displaystyle{\sup_{y_{1}} } \  f_{1}(y_{1}) |y_{1}-y_{3}|^{r_{13}+\theta}.$$
Therefore, together with (3.21) we conclude that for any small $\theta >0$,
$$\|f_{1}\|_{\frac{n}{r_{13}+\theta}, \infty} \|f_{2}\|_{\frac{n}{r_{12}+r_{23}-\theta}} \|f_{3}\|_{\infty} \lesssim A.  \eqno(3.22)$$
Meanwhile applying the similar arguments we have
$$\|f_{1}\|_{\infty}  \|f_{2}\|_{\frac{n}{r_{12}+r_{23}-\theta}} \|f_{3}\|_{\frac{n}{r_{13}+\theta}, \infty}  \lesssim A,$$
$$\|f_{1}\|_{\frac{n}{r_{12}+r_{13}-\theta}} \|f_{2}\|_{\infty} \|f_{3}\|_{\frac{n}{r_{23}+\theta}, \infty} \lesssim A, \
\|f_{1}\|_{\frac{n}{r_{12}+r_{13}-\theta}}  \|f_{2}\|_{\frac{n}{r_{23}+\theta}, \infty} \|f_{3}\|_{\infty}\lesssim A.  \eqno(3.23)$$
and
$$\|f_{1}\|_{\infty} \|f_{2}\|_{\frac{n}{r_{12}+\theta}, \infty} \|f_{3}\|_{\frac{n}{r_{13}+r_{23}-\theta}} \lesssim A, \
 \|f_{1}\|_{\frac{n}{r_{12}+\theta}, \infty} \|f_{2}\|_{\infty} \|f_{3}\|_{\frac{n}{r_{13}+r_{23}-\theta}} \lesssim A.  \eqno(3.24)$$
Since for all $ 0< p_{j} < \infty$, $1 \leq j \leq 3$ satisfying $\displaystyle{ \sum_{j=1}^{3}   }  \frac{1}{p_{j}}= \frac{1}{n} ( r_{12} +r_{13}+ r_{23})$
and $\frac{1}{p_{j}} < \frac{1}{n}     \displaystyle{\sum_ {i \neq j}}   r_{ij} $ , we can always find a small $\theta$ such that
 $(\frac{1}{p_{1}}, \frac{1}{p_{2}}, \frac{1}{p_{3}})$ lies in the interior of the convex hull of
 $(\frac{r_{13}+\theta}{n}, \frac{r_{12}+r_{23}-\theta}{n}, 0 )$, $(\frac{r_{12}+r_{13}-\theta}{n}, 0,  \frac{r_{23}+\theta}{n} )$,  $(0, \frac{r_{12}+\theta}{n}, \frac{r_{13}+r_{23}-\theta}{n} ) $, $(0,  \frac{r_{12}+r_{23}-\theta}{n}, \frac{r_{13}+\theta}{n} )$, $(\frac{r_{12}+r_{13}-\theta}{n}, \frac{r_{23}+\theta}{n}, 0 )$,
$( \frac{r_{12}+\theta}{n}, 0,   \frac{r_{13}+r_{23}-\theta}{n} ) $.
Similar to the discussion in Section 2 and Section 3,
inequality (3.20) follows immediately from (3.22)-(3.24) together with the following property
$$\|f\|_{q} \leq C_{p,q}  \ \|f\|_{p, \infty}^{\frac{p}{q}}  \|f\|_{\infty}^{1-\frac{p}{q}} $$
provided $0< p< q < \infty$.
\end{proof}

\bigskip

\begin{remark}
However, our method does not work for multilinear cases more than three functions.
Beckner \cite{Beckner}  gave a multilinear fractional integral inequality as follows,
mainly applying the general rearrangement inequality (Theorem 3.8  \cite{Lieb-Loss}) and the conformally invariant property of (3.27) below.  \\
For nonnegative functions $f_{j}\in L^{p_{j}}(\mathbb{R}^{n})$, $j=1, \dots, N$ and $p_{j}>1$, $\sum\limits_{j=1}^{N}\frac{1}{p_{j}}>1$.
Let $0 \leq r_{ij} = r_{j i} <n $  be real numbers satisfying
$$\displaystyle{ \sum_{j=1}^{N}   }  \frac{1}{p_{j}^{\prime}}= \frac{1}{n} \sum\limits_{1 \leq i<j \leq N} r_{ij}  \eqno(3.25), $$
and for every $j$
$$\frac{1}{p_{j}^{\prime}} = \frac{1}{2n}   \displaystyle{\sum_ {i \neq j}}   r_{ij} \eqno(3.26)$$
with $p_{j}$ and $p_{j}^{\prime}$ dual exponents. Then
$$\int_{(\mathbb{R}^{n})^{N}}  \prod\limits_{j=1}^{N}  f_{j}(y_{j}) \prod\limits_{1 \leq i < j \leq N} |y_{i}-y_{j}|^{-r_{ij}} dy_{1} \dots dy_{N}
\leq C_{p_{j},  r_{ij}, n , N} \prod\limits_{j=1}^{N} \|f_{j}\|_{p_{j}} \eqno(3.27)$$
Condition (3.25) follows from homogeneity. Condition (3.26) is to ensure conformal invariance of inequality (3.27).
Similarly to the arguments  in the alternative proof of part (a) of Theorem 3.1, we have the following theorem.

\end{remark}

\bigskip

\begin{theorem}
Let $ r_{ij} >0 $  and  $r_{i j} = r_{j i}$.
Let $f_{j}$  be nonnegative measurable functions defined on $\mathbb{R}^{n}$, $1 \leq j \leq N$. Then
$$\displaystyle{\prod_{j=1}^{N}} \|f_{j}\|_{p_{j}} \leq  C_{p_{j},  r_{ij}, n , N} \
 \displaystyle{\sup_{y_{j}}} \ \displaystyle{\prod_{j=1}^{N}} f_{j}(y_{j}) \prod\limits_{1 \leq i < j \leq N} |y_{i}-y_{j}|^{r_{ij}},  \eqno(3.28)$$
holds for any  $0 < p_{j} < \infty$ satisfying
$$\frac{1}{p_{j}}= \frac{1}{2n}   \displaystyle{\sum_ {i \neq j}} r_{ij}, \ \
\displaystyle{\sum_{j=1}^{N}}  \frac{1}{p_{j}}= \frac{1}{n} \sum\limits_{1 \leq i<j \leq N} r_{ij}.    \eqno(3.29)$$

\end{theorem}

\begin{proof}
For any $ r_{ij} >0 $, denote $\alpha=\displaystyle{\sum_ {i \neq j}} r_{ij}$,
then it is easy to see (3.28) is equivalent to the following inequality.
$$\displaystyle{\prod_{j=1}^{N}} \|f_{j}^{1/\alpha}\|_{p_{j} \alpha} \leq  C_{p_{j},  r_{ij}, n , N}^{1/\alpha} \
 \displaystyle{\sup_{y_{j}}} \ \displaystyle{\prod_{j=1}^{N}} f_{j}(y_{j})^{1/\alpha} \prod\limits_{1 \leq i < j \leq N} |y_{i}-y_{j}|^{\frac{r_{ij}}{\alpha}}.$$
Below it is enough to show that
$$\displaystyle{\prod_{j=1}^{N}} \|f_{j}\|_{p_{j} \alpha} \lesssim
 \displaystyle{\sup_{y_{j}}} \ \displaystyle{\prod_{j=1}^{N}}  f_{j}(y_{j}) \prod\limits_{1 \leq i < j \leq N} |y_{i}-y_{j}|^{\frac{r_{ij}}{\alpha}}. \eqno(3.30)$$
holds for any  $f_{j} \in L^{p_{j}\alpha}(\mathbb{R}^{n})$ with $0 < p_{j}\alpha < \infty$ satisfying
$$\frac{1}{p_{j}\alpha}= \frac{1}{2n}   \displaystyle{\sum_ {i \neq j}}  \frac{r_{ij}}{\alpha}, \ \
\displaystyle{\sum_{j=1}^{N}}  \frac{1}{p_{j}\alpha}= \frac{1}{n} \sum\limits_{1 \leq i<j \leq N} \frac{r_{ij}}{\alpha}.  \eqno(3.31)$$
Suppose $\displaystyle{\sup_{y_{j}}} \ \displaystyle{\prod_{j=1}^{N}} f_{j}(y_{j}) \prod\limits_{1 \leq i < j \leq N} |y_{i}-y_{j}|^{\frac{r_{ij}}{\alpha}} = A < \infty$.
We can write
\begin{align*}
& \ \ \ \ \ \|f_{1}\|_{p_{1}\alpha}^{p_{1}\alpha} \cdots\|f_{N}\|_{p_{N}\alpha}^{p_{N}\alpha} \\
&= \int_{(\mathbb{R}^{n})^{N}}  \displaystyle{\prod_{j=1}^{N}} f_{j}(y_{j})^{p_{j}\alpha} dy_{1} \dots dy_{N} \\
&= \int_{(\mathbb{R}^{n})^{N}} \displaystyle{\prod_{j=1}^{N}} f_{j}(y_{j}) \prod\limits_{1 \leq i < j \leq N} |y_{i}-y_{j}|^{\frac{r_{ij}}{\alpha}}
\displaystyle{\prod_{j=1}^{N}} f_{j}(y_{j})^{p_{j}\alpha-1} \prod\limits_{1 \leq i < j \leq N} |y_{i}-y_{j}|^{-\frac{r_{ij}}{\alpha}}
dy_{1} \dots dy_{N} \\
&\leq A \int_{(\mathbb{R}^{n})^{N}} \displaystyle{\prod_{j=1}^{N}} f_{j}(y_{j})^{p_{j}\alpha-1}
\prod\limits_{1 \leq i < j \leq N} |y_{i}-y_{j}|^{-\frac{r_{ij}}{\alpha}} dy_{1} \dots dy_{N}.
\end{align*}
Since every $p_{j}\alpha$ satisfies (3.31)  and  $\sum\limits_{1 \leq i<j \leq N} \frac{r_{ij}}{\alpha}=1$,
we have  $(p_{j}\alpha)^{\prime}>1$ and $\sum\limits_{j=1}^{N}\frac{1}{(p_{j}\alpha)^{\prime}}>1$.
This allows us to apply  inequality (3.27) to get
\begin{align*}
&\ \ \ \ \ \int_{(\mathbb{R}^{n})^{N}} \displaystyle{\prod_{j=1}^{N}} f_{j}(y_{j})^{p_{j}\alpha-1} \prod\limits_{1 \leq i < j \leq N} |y_{i}-y_{j}|^{-\frac{r_{ij}}{\alpha}} dy_{1} \dots dy_{N}  \\
&=\int_{(\mathbb{R}^{n})^{N}} \displaystyle{\prod_{j=1}^{N}} f_{j}(y_{j})^{\frac{p_{j}\alpha}{(p_{j}\alpha)^{\prime}}}
 \prod\limits_{1 \leq i < j \leq N} |y_{i}-y_{j}|^{-\frac{r_{ij}}{\alpha}} dy_{1} \dots dy_{N} \\
&\lesssim  \displaystyle{\prod_{j=1}^{N}} \|f_{j}^{\frac{p_{j}\alpha}{(p_{j}\alpha)^{\prime}}}\|_{(p_{j}\alpha)^{\prime}}
=    \displaystyle{\prod_{j=1}^{N}} \|f_{j}\|_{p_{j}\alpha}^{p_{j}\alpha-1}.
\end{align*}
Combining them together gives
$$\|f_{1}\|_{p_{1}\alpha}^{p_{1}\alpha} \cdots \|f_{k+1}\|_{p_{N}\alpha}^{p_{N}\alpha} \lesssim A \  \displaystyle{\prod_{j=1}^{N}} \ \|f_{j}\|_{p_{j}\alpha}^{p_{j}\alpha-1}.$$
This implies that
$$\displaystyle{\prod_{j=1}^{N}} \|f_{j}\|_{p_{j}\alpha} \lesssim  A
=\displaystyle{\sup_{y_{j}}} \ \displaystyle{\prod_{j=1}^{N}} f_{j}(y_{j}) \prod\limits_{1 \leq i < j \leq N} |y_{i}-y_{j}|^{\frac{r_{ij}}{\alpha}},$$
which gives (3.30).
Therefore by the equivalence as discussed above, this completes the proof of Theorem 3.7.

\end{proof}

\bigskip

\hspace{-13pt}{\bf Question.} Similarly  to the proof of Theorem 3.5, it is not hard to see the necessary conditions for inequality (3.28) to hold are  homogeneity condition and
 for every $j$, $1 \leq j \leq N$,
$$\frac{1}{p_{j}} < \frac{1}{n}  \displaystyle{\sum_ {i \neq j}}   r_{ij}. $$
We have already shown that it is sufficient for (3.28) to hold in the trilinear case together with the homogeneity condition.
An interesting problem is whether inequality (3.28) holds for any $p_{j}$ satisfying
$$\frac{1}{p_{j}} < \frac{1}{n}  \displaystyle{\sum_ {i \neq j}}   r_{ij},  \  \  \  \displaystyle{ \sum_{j=1}^{N} }  \frac{1}{p_{j}}= \frac{1}{n} \sum\limits_{1 \leq i<j \leq N} r_{ij},$$
where $N>3$.

\bigskip

\section{Sharp versions for geometric inequalities}

\medskip

\hspace{-13pt}{\it 1. Sharp constant for bilinear geometric inequality}\quad

\begin{theorem}
Let $0<p< \infty$ and $f$, $g$ be in  $ L^{p}(\mathbb{R}^{n})$.
For the geometric inequality
$$\|f\|_{L^{p}(\mathbb{R}^{n})}  \|g\|_{L^{p}(\mathbb{R}^{n})} \leq C_{p,n} \ \displaystyle{\sup_{x,y}} \ f(x)g(y)|x-y|^{\frac{2n}{p}}, \eqno(4.1)$$
the minimum constant $C_{p,n}$  is  obtained for $ f= const \cdot h$, and $ g= const \cdot h$,
where
$$h(x)=(1+|x|^{2})^{-\frac{n}{p}}.$$

\end{theorem}

Later we can see the sharp constant $C_{p,n}= 2^{-\frac{2n}{p}} |\mathbb{S}^{n}|^{\frac{2}{p}}$,
where $|\mathbb{S}^{n}|$ is the surface area of the unit sphere $\mathbb{S}^{n}$.

\medskip

Let $q \in (0, \infty)$, $p \in (1, \infty)$.
Form
$$\displaystyle{\sup_{x,y}} \ f(x)^{\frac{p}{q}} g(y)^{\frac{p}{q}} |x-y|^{\frac{2n}{q}}
= (\displaystyle{\sup_{x,y}} \ f(x) g(y)|x-y|^{\frac{2n}{p}})^{\frac{p}{q}}, $$
and
$$\|f^{\frac{p}{q}}\|_{L^{q}(\mathbb{R}^{n})}= ( \|f\|_{L^{p}(\mathbb{R}^{n})})^{\frac{p}{q}}, \
\|g^{\frac{p}{q}}\|_{L^{q}(\mathbb{R}^{n})}= ( \|g\|_{L^{p}(\mathbb{R}^{n})})^{\frac{p}{q}},$$
we observe that
if $f, g$ is a pair of extremals for $p \in (1, \infty)$, then
$f^{\frac{p}{q}}$, $g^{\frac{p}{q}}$ is a pair of extremals for any $q \in (0, \infty)$.
So it suffices to study the extremals for the case  when $1< p <\infty$.

 \medskip

In this section, we only consider such nonnegative measurable functions $f$, $g$ that the right hand side of (4.1) is finite.
For every nonnegative measurable function $f$, its  layer cake representation  is $f(x)= \int_{0}^{\infty} \chi_{\{f>t\}} (x) dt $,  where $ \chi_{\{f>t\}}$ is the characteristic function of the level set $ \{x: f(x)>t \}$.
For $A\subset \mathbb{R}^{n} $ of finite Lebesgue measure, we define the symmetric rearrangement of $A$ as $A^{\ast}:=\{x: |x|<r \} \equiv B (0, r)$ with $|A^{\ast}| =|A|$.
That is, $r^{n}=\frac{|A|}{v_{n}}$, and $v_{n}$ is the volume of unit ball in $\mathbb{R}^{n}$.
We then define the symmetric decreasing rearrangement of nonnegative measurable function $f$ as
$$\mathcal{R}f(x)=f^{\ast } (x):= \int_{0}^{\infty} \chi_{\{f>t\}^{\ast}} (x) dt,  $$
and define the Steiner symmetrisation of  $f$ with respect to the $j$-th coordinate as
$$\mathcal{R}_{j}f(x_{1}, \dots,x_{n})=f^{\ast j}(x_{1}, \dots,x_{n}) :=  \int_{0}^{\infty} \chi_{\{f(x_{1}, \dots ,x_{j-1}, \cdot, x_{j+1}, \dots , x_{n})>t\}^{\ast}}(x_{j}) dt.$$
We  observe that $f$ and $f^{\ast}$ are equimeasurable which means
$$|\{x: f(x) >t\} | = | \{x: f^{\ast}(x) > t\}|.$$
Together with the layer cake representation of $f$, hence $\|f\|_{p} = \|\mathcal{R}f\|_{p}$ for any $f\in L^{p}(\mathbb{R}^{n})$, $1\leq p \leq\infty$.
Besides,  $\|f\|_{p}= \| \mathcal{R}_{n} \dots \mathcal{R}_{1}f \|_{p}$ follows from Fubini's theorem.

We recall another related decreasing rearrangement of $f$ defined on $[0, \infty)$  as
$$f_{\ast}(t)= \inf \{\lambda > 0: m_{f}(\lambda) \leq t\},$$
where $m_{f}$ is the distrution function of $f$,
 $$m_{f}(\lambda) := |\{x\in \mathbb{R}^{n}: f(x)> \lambda \}|.$$
Then it is easy to see for any $x \in \mathbb{R}^{n}$,
$$f^{\ast}(x)= f_{\ast}(v_{n} |x|^{n} ).$$
As is well  known,  for $0 \leq s, t< \infty$
\begin{center}
$ f_{\ast}(s) > t$ \  if and only if \ $| \{x \in \mathbb{R}^{n}: f(x) >t \} | >s$.
\end{center}
By the relation of $f^{\ast}$  and $f_{\ast}$, we have for any $s\in \mathbb{R}^{n}$, $t \geq 0$
\begin{center}
$ f^{\ast}(s) > t$ \  if and only if \  $| \{x \in \mathbb{R}^{n}: f(x) >t \} | >v_{n} |s|^{n}$.
\end{center}

\bigskip

\begin{lemma}
Let $f, g$ be defined on $\mathbb{R}^{n}$, then
$$\displaystyle{\sup_{s,t}} f^{\ast}(s) g^{\ast}(t) |s-t| \leq \displaystyle{\sup_{x,y}} f(x) g(y) |x-y|.   \eqno(4.2) $$
\end{lemma}

\begin{proof}
Suppose $\displaystyle{\sup_{x,y}} f(x) g(y) |x-y|=A$.
We assume for a contradiction  that
$$\displaystyle{\sup_{s,t}} f^{\ast}(s) g^{\ast}(t) |s-t|  > A.$$
Then there exist  positive $\varepsilon$  and a set $ G \subset \mathbb{R}^{n} \times \mathbb{R}^{n}$ such that $|G|>0$ and
for all $(s_{0}, t_{0}) \in G$
we have
$$f^{\ast}(s_{0}) g^{\ast}(t_{0}) |s_{0}-t_{0} | > A+\varepsilon .$$
 It follows from $f^{\ast}(s_{0}) > (A+ \varepsilon) (g^{\ast}(t_{0}) | s_{0}-t_{0} |)^{-1}$ and the property of decreasing rearrangement discussed above  that
$$ | \{ x: f(x) > (A+\varepsilon) (g^{\ast}(t_{0}) | s_{0}-t_{0} |)^{-1} \} |> v_{n} |s_{0}|^{n}.   \eqno(4.3) $$
Denote the set $\{ x: f(x) > (A+\varepsilon) (g^{\ast}(t_{0}) | s_{0}-t_{0} |)^{-1} \}$ by $E$,
so $$g^{\ast}(t_{0}) >( A+\frac{\varepsilon}{2}) \  (\displaystyle{\inf_{x\in E }} f(x) | s_{0}-t_{0} | )^{-1}.$$
Applying the property of decreasing rearrangement again, we have
$$| \{ y: g(y) >  ( A+\frac{\varepsilon}{2}) (\displaystyle{\inf_{x\in E }} f(x) | s_{0}-t_{0} | )^{-1}  \} |> v_{n} |t_{0}|^{n}.   \eqno(4.4)$$
Denote the set $\{ y: g(y) >( A+\frac{\varepsilon}{2}) (\displaystyle{\inf_{x\in E }} f(x) | s_{0}-t_{0} | )^{-1} \}$ by $F$.
Then  $s_{0}\in E^{\ast}, t_{0}\in F^{\ast}$.
It turns out that
$$ \displaystyle{\sup_{x\in E, y\in F}}  |x-y| \geq | s_{0}-t_{0} |.     \eqno(4.5)$$
The reason is as follows. In the first place, it is easy to observe for any measurable set $C\subset \mathbb{R}^{n} $
$$ \displaystyle{\sup_{x\in C}}  \ |x| \geq   \displaystyle{\sup_{x\in C^{\ast}}} \  |x|.$$
If $\displaystyle{\sup_{x\in C}} |x| < \displaystyle{\sup_{x\in C^{\ast}}} |x| \equiv s$,
there exist positive $\delta$ and a measure zero set $M \subset \mathbb{R}^{n}$, such that
$|x| < s-\delta$ for any $x \in C \setminus M$.
So $$C \setminus M \subset B(0, s-\delta),$$
where $B(0, s-\delta)$ is the ball centred at $0$ with radius $s-\delta$.
This shows $(C \setminus M)^{\ast} = C^{\ast} \subset B(0, s-\delta)$, which is a contradiction.

\medskip

Hence based on the property  $\displaystyle{\sup_{x\in C}} \  |x| \geq   \displaystyle{\sup_{x\in C^{\ast}}}  \ |x|$ for any  set $C\subset \mathbb{R}^{n}$ of finite Lebesgue measure, we have
$$\displaystyle{\sup_{x\in E, y\in F}}  |x-y| =  \displaystyle{\sup_{z\in E- F}} |z|   \geq   \displaystyle{\sup_{z\in (E- F)^{\ast}}} |z|.$$
The Brunn-Minkowski inequality tells for measurable sets with finite volume $E$ and $F$,
$$|E-F|^{1/n} \geq |E|^{1/n}+|F|^{1/n}. $$
By the definition of symmetric rearrangement of $E$ and $F$, we have
$$ E^{\ast} = B(0, r_{1} ), \ \ F^{\ast} = B(0, r_{2}), $$
where their radius are $r_{1}=(\frac{|E|}{v_{n}})^{1/n}$,  $r_{2}= (\frac{|F|}{v_{n}})^{1/n} $ respectively. \\
Then $E^{\ast}+F^{\ast}$ is the ball centred at $0$ with radius $r_{1}+r_{2}$, and
$$E^{\ast}- F^{\ast} = E^{\ast}+ F^{\ast} = B(0,  r_{1}+r_{2} ).$$
Together with  the Brunn-Minkowski inequality, we have
\begin{align*}
 |(E-F)^{\ast}|^{1/n} &=|E-F|^{1/n}  \geq |E|^{1/n}+|F|^{1/n}  \\
&=|E^{\ast}|^{1/n}+|F^{\ast}|^{1/n}\\
&= v_{n}^{1/n} r_{1} + v_{n}^{1/n} r_{2},
\end{align*}
which means
$$|(E-F)^{\ast}| \geq v_{n} (r_{1} + r_{2})^{n}=  | E^{\ast}+F^{\ast}|.$$
Therefore  $$s_{0}- t_{0}\in  E^{\ast}-F^{\ast}=  E^{\ast}+F^{\ast} \subset  (E-F)^{\ast}.$$
Moreover
$$ \displaystyle{\sup_{x\in E, y\in F}}  |x-y| =   \displaystyle{\sup_{z\in E- F}} |z|  \geq   \displaystyle{\sup_{z\in (E- F)^{\ast}}} |z|  \geq   \displaystyle{\sup_{x\in E^{\ast},  y \in F^{\ast} }}  |x-y|  \geq | s_{0}-t_{0} |, \eqno(4.6) $$
which completes the proof of  (4.5).

\medskip

 Now (4.4)  implies that for any $x \in E$, $y\in F$
$$ f(x) g(y) |x-y|  > (A + \frac{\varepsilon}{2} ) | s_{0}-t_{0} | ^{-1} |x-y|,$$
thus
$$ \displaystyle{\sup_{x\in E, y\in F}} f(x) g(y) |x-y|  \geq   (A + \frac{\varepsilon}{2} ) | s_{0}-t_{0} | ^{-1}    \displaystyle{\sup_{x\in E, y\in F}} |x-y|.$$
Consequently, together with (4.5) we get
\begin{align*}
\displaystyle{\sup_{x,y}} f(x) g(y) |x-y|
& \geq   (A + \frac{\varepsilon}{2} ) | s_{0}-t_{0} | ^{-1}    \displaystyle{\sup_{x\in E, y\in F}} |x-y| \\
& \geq   (A + \frac{\varepsilon}{2} ) | s_{0}-t_{0} | ^{-1}   | s_{0}-t_{0} |  \\
&>A.
\end{align*}
This  is a contradiction.

\end{proof}

\medskip

However, we do not know when there is equality in (4.2). One might guess that strict inequality (4.2) holds only if
$f(x)=f^{\ast}(x-y)$ and $g(y)=f^{\ast}(x-y)$ for some $y$ in $\mathbb{R}^{n}$.
By the following counterexample, we show that this is not true.
In the one-dimensional case, let
$$f(x)=4 \chi_{|x| \leq |E_{1}|}+ \chi_{|E_{1}| < x \leq |E_{1}|+ 2|E_{2}|}$$
with $|E_{1}|>|E_{2}|$, and $f=g$.
Then
$$f^{\ast}(x)=4 \chi_{|x| \leq |E_{1}|}+ \chi_{|E_{1}| <|x| \leq  |E_{1}|+ |E_{2}|}.$$
It is easy to check that
$$\displaystyle{\sup_{x,y}} f(x) g(y) |x-y|= \max \{32|E_{1}|, 8(|E_{1}|+|E_{2}|), 2|E_{2}| \}= 32|E_{1}|,$$
and
$$\displaystyle{\sup_{x,y}} f^{\ast}(x) g^{\ast}(y) |x-y| = \max \{32|E_{1}|, 4(2|E_{1}|+|E_{2}|), 2(|E_{1}|+|E_{2}|) \}= 32|E_{1}|.$$
So there are other classes of examples where equality holds.

\bigskip

Due to  Lemma 4.2, it suffices to seek   optimisers amongst the class of  all symmetric decreasing functions.

Let $\mathcal{S}$ be the stereographic projection from $\mathbb{R}^{n}$ to the unit sphere $\mathbb{S}^{n}$  with
$$\mathcal{S}(x)=(\frac{2x_{1}}{1+|x|^{2}},  \dots, \frac{2x_{n}}{1+|x|^{2}}, \frac{1-|x|^{2}}{1+|x|^{2}}),$$
where $ x=( x_{1}, \dots, x_{n}) \in \mathbb{R}^{n}$.
So $$\mathcal{S}^{-1} (s) = (\frac{s_{1}}{1+s_{n+1}}, \dots, \frac{s_{n}}{1+s_{n+1}} ).$$
For $f\in L^{p}(\mathbb{R}^{n})$, define
$$(\mathcal{S}^{\ast}f)(s):= |J_{\mathcal{S}^{-1}}(s)|^{1/p}f(\mathcal{S}^{-1}(s)),  \  (\mathcal{S}^{\ast}g)(t):=|J_{\mathcal{S}^{-1}}(t)|^{1/p}g(\mathcal{S}^{-1}(t)), \eqno(4.7)$$
where $J_{\mathcal{S}^{-1}}$ is the Jacobian determinant of the map $\mathcal{S}^{-1}$
$$ |J_{\mathcal{S}^{-1}}(s)| = (\frac{1}{1+ s_{n+1}})^{n}= \frac{1}{2^{n}}  (1+|\mathcal{S}^{-1} (s)|^{2})^{n}.    \eqno(4.8)$$
Then we have the invariance of the geometric inequality under the stereographic projection shown as the following lemma.

\medskip

\begin{lemma}
For $f, g \in L^{p}(\mathbb{R}^{n})$, denote $F(s)=(\mathcal{S}^{\ast}f)(s), G(t)= (\mathcal{S}^{\ast}g)(t)$.  Then
$$\displaystyle{\sup_{x,y\in \mathbb{R}^{n}}} f(x) g(y) |x-y|^{\frac{2n}{p}} =\displaystyle{\sup_{s,t\in \mathbb{S}^{n}}} F(s) G(t) |s-t|^{\frac{2n}{p}},$$
and
$$\|F\|_{L^{p}(\mathbb{S}^{n})}= \|f\|_{L^{p}(\mathbb{R}^{n})},  \ \|G\|_{L^{p}(\mathbb{S}^{n})} = \|g\|_{L^{p}(\mathbb{R}^{n})}.$$

\end{lemma}

\begin{proof}
By the stereographic projection $\mathcal{S}$ ,
we have (4.8)
$$| J_{\mathcal{S}^{-1}}(s)|=( \frac{1+|x|^{2}}{2})^{n},$$
and let $x=\mathcal{S}^{-1} (s), y=\mathcal{S}^{-1}(t)$, then
$$|x-y|=|s-t| (\frac{1+|x|^{2}}{2})^{1/2} (\frac{1+|y|^{2}}{2})^{1/2}= |J_{\mathcal{S}^{-1}}(s)|^{\frac{1}{2n}}|J_{\mathcal{S}^{-1}}(t)|^{\frac{1}{2n}} |s-t|.$$
So
\begin{align*}
\displaystyle{\sup_{s,t\in \mathbb{S}^{n}}} F(s) G(t) |s-t|^{\frac{2n}{p}}&= \displaystyle{\sup_{s,t\in \mathbb{S}^{n}}} |J_{\mathcal{S}^{-1}}(s)|^{1/p}f(\mathcal{S}^{-1}(s))
|J_{\mathcal{S}^{-1}}(t)|^{1/p}g(\mathcal{S}^{-1}(t)) |s-t|^{\frac{2n}{p}}  \\
&=\displaystyle{\sup_{x,y\in \mathbb{R}^{n}}}  f(x) g(y) |x-y|^{\frac{2n}{p}}.
\end{align*}
The invariance of $L^{p}$ norm can be obtained as follows,
$$\|f\|_{L^{p}(\mathbb{R}^{n})} = (\int_{\mathbb{R}^{n}} |f(x)|^{p} dx)^{1/p}
= (\int_{\mathbb{S}^{n}} |f(\mathcal{S}^{-1}(s))|^{p} |J_{\mathcal{S}^{-1}}(s)| ds)^{1/p}
= (\int_{\mathbb{S}^{n}} |F(s)|^{p} ds)^{1/p}.$$
Applying a similar argument implies $\|g\|_{L^{p}(\mathbb{R}^{n})}=\|G\|_{L^{p}(\mathbb{S}^{n})}$.

\end{proof}

Now we turn to study the sharp  case of inequality (4.1).

\bigskip

\hspace{-13pt}{\it Proof of Theorem 4.1}\quad

For $f\in L^{p}(\mathbb{R}^{n})$,  consider a  rotation $D: \mathbb{S}^{n} \rightarrow \mathbb{S}^{n}$ with
$$D(s)=(s_{1}, \dots,s_{n-1}, s_{n+1}, -s_{n}).$$
Specifically,  it is a rotation of the sphere by $90^{\circ}$ which keeps the other basis vectors fixed except $n$-th and $(n+1)$-th vectors in the direction of mapping
the  $(n+1)$-th vector $e_{n+1}= (0, \dots,0,1) $ to $n$-th vector $e_{n}=(0, \dots,0,1,0)$. \\
Define $(D^{\ast}F)(s)= |J_{D^{-1}}(s)|^{\frac{1}{p}} F(D^{-1}(s)) = F(D^{-1}(s)) $ for any $F \in L^{p}(\mathbb{S}^{n})$.
Then
$$\|  D^{\ast} F \|_{p}= \| F \|_{p},$$
which shows $D^{\ast}$ is norm preserving.

We consider the new function $(\mathcal{S}^{\ast})^{-1} D^{\ast} \mathcal{S}^{\ast} f$, where
$(\mathcal{S}^{\ast}f)(s)$  is the same as $ (4.7)$.
Denote $ (\mathcal{S}^{\ast}f)(s)$  by $ F(s)$, and let $x= \mathcal{S}^{-1} (s)$.
 From the discussion above, we have already shown
 $$F(s)= (\frac{1+|x|^{2}}{2})^{\frac{n}{p}} f(x).$$
The definition of $D$ and $ \mathcal{S}$ implies
$$D^{-1} (s)=(s_{1}, \dots,s_{n-1}, -s_{n+1}, s_{n})=( \frac{2x_{1}}{1+|x|^{2}}, \dots, \frac{2x_{n}}{1+|x|^{2}}, \frac{|x|^{2}-1}{1+|x|^{2}}, \frac{2x_{n}}{1+|x|^{2}} ).$$
Then
$$ (D^{\ast} \mathcal{S}^{\ast} f)(s)= (D^{\ast}F)(s)= F(D^{-1}(s))= (\frac{1+|x|^{2}}{|x+e_{n}|})^{\frac{n}{p}} f(  \frac{2x_{1}}{|x+e_{n}| ^{2}}, \dots,  \frac{2x_{n-1}}{|x+e_{n}|^{2}},  \frac{|x|^{2}-1}{|x+e_{n}|^{2}} ), $$
this is because  $$\mathcal{S}^{-1} (D^{-1} (s)) = ( \frac{2x_{1}}{|x+e_{n}| ^{2}}, \dots,  \frac{2x_{n-1}}{|x+e_{n}|^{2}},  \frac{|x|^{2}-1}{|x+e_{n}|^{2}} ).$$
Finally we find
\begin{align*}
( \mathcal{S}^{\ast-1} D^{\ast} \mathcal{S}^{\ast}f ) (x) &=  ( \frac{1+|x|^{2}}{2})^{-\frac{n}{p}} F(D^{-1}(s)) \\
&= (\frac{2}{|x+e_{n}|})^{\frac{n}{p}}  f(  \frac{2x_{1}}{|x+e_{n}| ^{2}}, \dots,  \frac{2x_{n-1}}{|x+e_{n}|^{2}},  \frac{|x|^{2}-1}{|x+e_{n}|^{2}} ).
\end{align*}
Briefly speaking,  we lift $f$ to the sphere by (4.7) first, then rotate it by $90^{\circ}$ in a specific direction which maps the north pole $e_{n+1}$ to $ e_{n}$, lastly push back to $\mathbb{R}^{n}$.
For simplicity we denote $\mathcal{S}^{\ast -1} D^{\ast} \mathcal{S}^{\ast}f$ by $\mathcal{D}f$.

\bigskip

Let $f\in L^{p}(\mathbb{R}^{n})$. Applying the transformation $\mathcal{D}$ and the symmetric rearrangement to $f$ many times gives the sequence $\{f_{k}\}_{k \in \mathbb{N}}$.
Specifically, $f_{0}=f$, $f_{k} =(\mathcal{R}\mathcal{D})^{k}f$ .   Note that both $\mathcal{D}$ and $\mathcal{R}$ are norm-preserving.
This is because  Lemma 4.3 implies
$$ \| \mathcal{S}^{\ast -1} D^{\ast} \mathcal{S}^{\ast}f \|_{p} = \|   D^{\ast} \mathcal{S}^{\ast}f  \|_{p}.$$
Due to the  norm preserving property of $ D^{\ast}$, we have
$$ \|   D^{\ast} \mathcal{S}^{\ast}f  \|_{p} = \|   \mathcal{S}^{\ast}f  \|_{p}$$
So apply  Lemma 4.3 again to get
$$ \| \mathcal{S}^{\ast -1} D^{\ast} \mathcal{S}^{\ast}f \|_{p} = \|   \mathcal{S}^{\ast}f  \|_{p}= \|   f  \|_{p}.$$
 It follows from Theorem 4.6 in Lieb-Loss  \cite{Lieb-Loss}  that  for all $ f \in L^{p}(\mathbb{R}^{n})$, the sequence
 $f_{k}$ converges to $h_{f}$  in $L^{p}$ norm as $k \rightarrow \infty$.
Here
$$ h_{f}= c  \ h, \ \  h(x)=(1+|x|^{2})^{-\frac{n}{p}}$$
and $c$ is the constant such that $\|f\|_{p} =  \|h_{f}\|_{p}$,
so  the constant $c$ is
$$c=2^{\frac{n}{p}} |\mathbb{S}^{n}|^{-1/p}   \|f\|_{p} ,$$
where $|\mathbb{S}^{n}|$ means the area of unit sphere in $\mathbb{R}^{n+1}$.

Since $f_{k}$  converges to $ h_{f}$  in $L^{p}$ norm for all f $ \in L^{p}(\mathbb{R}^{n})$,
 there exist subsequences $\{ f_{k_{l}} \},  \{ g_{k_{l}} \}$ such that
$ f_{k_{l}} \rightarrow  h_{f}$ and  $ g_{k_{l}} \rightarrow  h_{g}$ pointwise almost everywhere as $l  \rightarrow \infty$.
Clearly, Lemma 4.2, Lemma 4.3 and the rearrangement property
\begin{center}
$(f^{p})^{\ast}= (f^{\ast})^{p}$, \ \ for \ $0<p<\infty$
\end{center}
indicate that
 $ \displaystyle{\sup_{x,y}}  f_{k}(x)  g_{k}(y) |x-y|^{\frac{2n}{p}}  $  decreases monotonically  as $k$ grows.
Hence for all  $ x, y, k_{l}$
$$  f_{k_{l}}(x)  g_{k_{l}}(y) |x-y|^{\frac{2n}{p}}  \leq  \displaystyle{\sup_{x,y}} f(x) g(y)  |x-y|^{\frac{2n}{p}} < \infty.$$
Together with the dominated convergence theorem  it follows that
$$f_{k_{l}}(x)  g_{k_{l}}(y) |x-y|^{\frac{2n}{p}}  \xrightarrow{weak^* }   h_{f}(x)  h_{g}(y)  |x-y|^{\frac{2n}{p}} $$
 in $ L^{\infty}(\mathbb{R}^{n} \times  \mathbb{R}^{n}  )$  as $l  \rightarrow \infty$. \\
Hence  by the $\mathrm{weak}^*$ lower semicontinuity of the $L^{\infty}$ norm we have
\begin{align*}
\displaystyle{\sup_{x,y}} \ h_{f}(x)  h_{g}(y)  |x-y|^{\frac{2n}{p}}
&\leq \displaystyle{\liminf_{l}} \  (\displaystyle{\sup_{x,y}} \ f_{k_{l}}(x)  g_{k_{l}}(y) |x-y|^{\frac{2n}{p}} ) \\
&=  \displaystyle{\inf_{l}} \  (\displaystyle{\sup_{x,y}}  \ f_{k_{l}}(x)  g_{k_{l}}(y) |x-y|^{\frac{2n}{p}} ).
\end{align*}
Therefore  for every $f, g \in L^{p}(\mathbb{R}^{n})$ and every $k_{l}$
$$ \frac{ \displaystyle{\sup_{x,y}} \  f(x) g(y)  |x-y|^{\frac{2n}{p}} }{  \|f \|_{p}    \|g\|_{p}  }   \geq   \frac{  \displaystyle{\sup_{x,y}} \ f_{k_{l}}(x)  g_{k_{l}}(y) |x-y|^{\frac{2n}{p}} }{ \|f_{k_{l}} \|_{p}    \|g_{k_{l}}\|_{p}}
     \geq  \frac{ \displaystyle{\sup_{x,y}} \  h_{f}(x)  h_{g}(y)  |x-y|^{\frac{2n}{p}}      }{ \|h_{f} \|_{p}    \|h_{g}\|_{p}},$$
because of the norm-preserving  of $\mathcal{R}\mathcal{D}$.\\
Obviously,
$$\frac{ \displaystyle{\sup_{x,y}} \  h_{f}(x)  h_{g}(y)  |x-y|^{\frac{2n}{p}}      }{ \|h_{f} \|_{p}    \|h_{g}\|_{p}}
= \frac{ \displaystyle{\sup_{x,y}} \  h(x)  h(y)  |x-y|^{\frac{2n}{p}}      }{ \|h\|_{p}^{2}  }.$$

$\hfill\Box$

\medskip

Therefore, the conformally invariant property of (4.1) implies that  if $f$ and $g$ are the same conformal transformation of $h$, equality still holds.
However,  here we can not characterise  the optimisers.

\bigskip

From the sharp version for $\mathbb{R}^{n}$ case in Theorem 4.1
together with (4.7), (4.8) and the conformally invariant property in Lemma 4.3,
it follows that the geometric inequality (4.1) has conformally equivalent
form on the unit sphere $\mathbb{S}^{n}$ as follows.

\medskip

\begin{theorem}
For $0 < p < \infty$, let $F, G$ be nonnegative functions in $L^{p}(\mathbb{S}^{n})$. Then
$$\|F\|_{L^{p}(\mathbb{S}^{n})} \ \|G\|_{L^{p}(\mathbb{S}^{n})} \leq B_{p,n} \
\displaystyle{\sup_{s,t\in \mathbb{S}^{n}}}  \  F(s) G(t) \ |s-t|^{\frac{2n}{p}}.   \eqno(4.9)$$
The best constant $B_{p,n}$  is  obtained for
$F$,  $G$ are constant functions,
and the corresponding $B_{p,n} = 2^{-\frac{2n}{p}} |\mathbb{S}^{n}|^{\frac{2}{p}}$.

\end{theorem}

\bigskip

Meanwhile, let $\mathbb{H}^{n}$ be the hyperbolic space in $\mathbb{R}^{n+1}$:
$$\mathbb{H}^{n}= \{q=(q_{1}, \dots,q_{n}, q_{n+1}) \in \mathbb{R}^{n}\times \mathbb{R}: q_{1}^{2}
+ \dots+ q_{n}^{2}- q_{n+1}^{2}= -1 )\},$$
with the Lorenz group $O(1,n)$ invariant measure $d \nu (q)$.
We find the geometric inequality (4.1) also has the conformally equivalent form in
$\mathbb{H}^{n}$ space as shown in the following theorem.

\medskip

\begin{theorem}
For $0 < p < \infty$, let $F, G$  be nonnegative functions in  $L^{p}(\mathbb{H}^{n})$. Then
$$\|F\|_{L^{p}(\mathbb{H}^{n})} \ \|G\|_{L^{p}(\mathbb{H}^{n})}
\leq E_{p, n} \  \displaystyle{\sup_{q,t}}  \ F(q) G(t) \ |qt-1|^{\frac{n}{p}},  \eqno(4.10)$$
$qt= -q_{1}t_{1}- \dots-q_{n}t_{n}+q_{n+1}t_{n+1}$.
The best constant $E_{p, n}$  is  obtained when
$F= const \cdot H$, $G= const \cdot H$, where
$$H(q)= |q_{n+1}|^{-\frac{n}{p}}, \ q=(q_{1}, \dots,q_{n}, q_{n+1}) .$$

\end{theorem}

\begin{proof}
Consider the stereographic projection $\mathcal{H}$ which is conformal transformation
from $\mathbb{R}^{n} \backslash \{|x|=1\}$ to $\mathbb{H}^{n}$ as
$$\mathcal{H}(x)=(\frac{2x_{1}}{1-|x|^{2}}, \dots,\frac{2x_{n}}{1-|x|^{2}}, \frac{1+|x|^{2}}{1-|x|^{2}}),$$
so
$$\mathcal{H}^{-1}(q)= (\frac{q_{1}}{1+q_{n+1}}, \dots,\frac{q_{n}}{1+q_{n+1}}).$$
The Jacobian determinant of the map $\mathcal{H}^{-1}$ is
$$|J_{\mathcal{H}^{-1}}(q)|= (\frac{1- |\mathcal{H}^{-1}(q)|^{2}}{2})^{n}.$$
Let $x=\mathcal{H}^{-1}(q), y=\mathcal{H}^{-1}(t)$,
then we have
$$|x-y|= (\frac{|1-|x|^{2}|}{2})^{1/2} (\frac{|1-|y|^{2}|}{2})^{1/2} |qt-1|^{1/2}=|J_{\mathcal{H}^{-1}}(q)|^{\frac{1}{2n}}|J_{\mathcal{H}^{-1}}(t)|^{\frac{1}{2n}} |qt-1|^{1/2},$$
where
$qt= -q_{1}t_{1}- \dots -q_{n}t_{n}+q_{n+1}t_{n+1}$. \\
Define
$$F(q):= |J_{\mathcal{H}^{-1}}(q)|^{1/p} f(\mathcal{H}^{-1}(q)),  \ G(t):= |J_{\mathcal{H}^{-1}}(t)|^{1/p} g(\mathcal{H}^{-1}(t)).$$
Thus from the above, we easily get the conformal invariance as follows.
\begin{align*}
\displaystyle{\sup_{q,t}}  \ F(q) G(t) |qt-1|^{\frac{n}{p}}
&=\displaystyle{\sup_{q,t}} \ |J_{\mathcal{H}^{-1}}(q)|^{1/p} f(\mathcal{H}^{-1}(q)) |J_{\mathcal{H}^{-1}}(t)|^{1/p} g(\mathcal{H}^{-1}(t))|qt-1|^{\frac{n}{p}} \\
&=\displaystyle{\sup_{x,y}} \  f(x) g(y) |x-y|^{\frac{2n}{p}},
\end{align*}
and
$$\|f\|_{L^{p}(\mathbb{R}^{n})}=(\int_{\mathbb{R}^{n}} |f(x)|^{p} dx)^{1/p}
=(\int_{\mathbb{H}^{n}} |f(\mathcal{H}^{-1}(q))|^{p} |J_{\mathcal{H}^{-1}}(q)| dq)^{1/p}
=(\int_{\mathbb{H}^{n}} |F(q)|^{p} dq)^{1/p} .$$
Applying a similar argument gives  $\|g\|_{L^{p}(\mathbb{R}^{n})}=\|G\|_{L^{p}(\mathbb{H}^{n})}$.

When $f(x)=c (1+|x|^{2})^{-\frac{n}{p}}$, $F(q)= c \ (\frac{1-|x|^{2}}{1+|x|^{2}})^{\frac{n}{p}}
= c \ |q_{n+1}|^{-\frac{n}{p}}$.
Hence the conformally equivalent form  (4.10) follows from Theorem 4.1.

\end{proof}

\bigskip

\hspace{-13pt}{\it 2.  Sharp constant for  multilinear geometric inequality}\quad

\begin{theorem}
Let $0<p< \infty$ and $f_{j}$  be in  $ L^{p}(\mathbb{R}^{n})$.
For multilinear geometric  inequality
$$  \displaystyle{\prod_{j=1}^{n+1}} \|f_{j}\|_{L^{p}(\mathbb{R}^{n})} \leq C_{p,n} \
\displaystyle{\sup_{y_{j}}} \ \displaystyle{\prod_{j=1}^{n+1}} \ f_{j}(y_{j}) \det(y_{1}, \dots,y_{n+1})^{\gamma}  \eqno(4.11)   $$
with $ \gamma =\frac{n+1}{p}$,
the minimum constant is obtained when $ f_{j}= const \cdot h$, $1 \leq j \leq n+1$,
where
$$h (x)=({1+|x|^{2}})^{-\frac{n+1}{2p}}.$$

\end{theorem}

Later we can see the sharp constant $C_{p,n}=(\frac{1}{2}|\mathbb{S}^{n}| )^{\frac{n+1}{p}}$,
where $|\mathbb{S}^{n}|$ is the surface area of the unit  sphere $\mathbb{S}^{n}$.

 \medskip

As before, it suffices to study the extremals for the case  when $1< p <\infty$.
Because for any $q \in (0, \infty)$, $p \in (1, \infty)$, we have
$$\displaystyle{\sup_{y_{j}}} \ \displaystyle{\prod_{j=1}^{n+1}} \ f_{j}(y_{j})^{\frac{p}{q}} \det(y_{1}, \dots,y_{n+1})^{\frac{n+1}{q}}
=(\displaystyle{\sup_{y_{j}}} \ \displaystyle{\prod_{j=1}^{n+1}} \ f_{j}(y_{j}) \det(y_{1}, \dots,y_{n+1})^{\frac{n+1}{p}})^{\frac{p}{q}}.$$
and for each $j$,
$$\|f_{j}^{\frac{p}{q}}\|_{L^{q}(\mathbb{R}^{n})}= ( \|f_{j}\|_{L^{p}(\mathbb{R}^{n})})^{\frac{p}{q}}$$
Thus if $ \{f_{j}\}$ are the extremal functions for $p \in (1, \infty)$, then
$\{f_{j}^{\frac{p}{q}}\}$ are the  extremal functions for any $q \in (0, \infty)$.

\bigskip

\begin{lemma}
Let $a_{j} \in \mathbb{R}$ and $A_{j}$ be sets in $\mathbb{R}$ with finite measure, $j=1, \dots, l$. Then
$$\displaystyle{\sup_{x_{j}\in A_{j}^{\ast} }}  \  |\sum_{j=1}^{l} a_{j}x_{j}|  \leq \displaystyle{\sup_{x_{j}\in A_{j} }}  \  |\sum_{j=1}^{l} a_{j}x_{j}|. \eqno(4.12)$$

\end{lemma}

\begin{proof}
For simplicity we just show
$$\displaystyle{\sup_{x\in A^{\ast},  y\in B^{\ast},  z\in C^{\ast}}}|a x+ b y + c z|  \leq  \displaystyle{\sup_{x\in A, y\in B, z\in C}}|a x+ b y + c z|,\eqno(4.13)$$
where $a,b,c \in \mathbb{R}$, $A, B, C\subset  \mathbb{R}$ with finite measure.
Then (4.12) can be  obtained by induction on $l$.

The definition of symmetric rearrangement of set implies
$$ A^{\ast} = (-\frac{|A|}{2},   \frac{|A|}{2}), \  B^{\ast} = (-\frac{|B|}{2},   \frac{|B|}{2}),$$
so
$$A^{\ast}+B^{\ast}= (-\frac{|A|+|B|}{2}, \frac{|A|+|B|}{2}).$$
The Brunn-Minkowski inequality tells $|A+B| \geq |A|+|B|$,
then
$$  |(A+B)^{\ast}| = |A+B|  \geq |A|+|B| = |A^{\ast} +B^{\ast}|$$
which means
$$(A+B)^{\ast} \supset A^{\ast}+B^{\ast}.$$
Together with the property $(aA)^{\ast}=a A^{\ast}$  we have
\begin{align*}
\displaystyle{\sup_{x\in A^{\ast}, y\in B^{\ast}, z\in C^{\ast}}}|a x+ b y + c z|
&= \displaystyle{\sup_{\overline{x}\in a A^{\ast}, \overline{y}\in b B^{\ast}, \overline{z}\in c C^{\ast} } }  | \overline{x}+ \overline{y} + \overline{z} | \\
&=   \displaystyle{\sup_{\overline{m}\in a A^{\ast} + b B^{\ast}, \overline{z}\in c C^{\ast}} }| \overline{m}+ \overline{z} | \\
&\leq \displaystyle{\sup_{\overline{m}\in (aA+bB)^{\ast}, \overline{z}\in c C^{\ast}} }| \overline{m}+ \overline{z} |.
\end{align*}

In the proof of Lemma 4.2  we stated that
$$\displaystyle{\sup_{x\in A^{\ast}, y\in B^{\ast}}} | x + y | \leq \displaystyle{\sup_{x\in A, y\in B}} | x + y |.$$
Hence
\begin{align*}
\displaystyle{\sup_{\overline{m}\in (aA+bB)^{\ast}, \overline{z}\in c C^{\ast}}} | \overline{m}+ \overline{z} |
&\leq \displaystyle{\sup_{m\in aA+bB, z \in c C }}| m + z| \\
&= \displaystyle{\sup_{x\in A, y\in B, z\in C}} |a x+ b y + c z|
\end{align*}
which gives (4.13).

\end{proof}

\bigskip

\begin{lemma} [the general form of Lemma 4.2]
Let $f_{j}$ be defined on $\mathbb{R}$ and let  $ a_{j}$ be real numbers, then
$$\displaystyle{\sup_{x_{j} }} \prod_{j=1}^{l} f_{j}^{\ast}(x_{j})|\sum_{j}^{l} a_{j}x_{j}|  \leq
 \displaystyle{\sup_{x_{j} }}  \prod_{j=1}^{l} f_{j}(x_{j})|\sum_{j}^{l} a_{j}x_{j}|.  \eqno(4.14)$$

\end{lemma}

\begin{proof}
For simplicity, we just see why
$$\displaystyle{\sup_{x, y, z}} \  f^{\ast}(x) g^{\ast}(y) h^{\ast}(z) |a x+ b y + c z| \leq  \displaystyle{\sup_{x, y, z} }  \ f(x)g(y)h(z)|a x+ b y + c z|;   \eqno(4.15)$$
$(4.14)$ can be obtained by induction  and similar arguments as the proof of $(4.15)$.
Suppose for a contradiction that
$$\displaystyle{\sup_{x, y, z }} f^{\ast}(x)g^{\ast}(y)h^{\ast}(z)|a x+ b y + c z| > \displaystyle{\sup_{x, y, z}}  f(x)g(y)h(z)|a x+ b y + c z| = \alpha,$$
which implies there exist positive $\varepsilon $ and a set $G \subset \mathbb{R} \times  \dots \times \mathbb{R}$  such that $|G| >0$, and
 $$f^{\ast}(\overline{x}) g^{\ast}(\overline{y}) h^{\ast}(\overline{z}) |a \overline{x}+ b \overline{y} + c \overline{z}| > \alpha +\varepsilon $$
for all $(\overline{x}, \overline{y}, \overline{z}) \in G$.
So
$$f^{\ast}(\overline{x}) > (\alpha +\varepsilon) (g^{\ast}(\overline{y})h^{\ast}(\overline{z})|a \overline{x}+ b \overline{y} + c \overline{z}|)^{-1}. $$

By the property of decreasing  rearrangement, we have
$$|E_{1}|\equiv | \{x: f(x) > (\alpha +\varepsilon) (g^{\ast}(\overline{y}) h^{\ast}(\overline{z})|a \overline{x}+ b \overline{y} + c \overline{z}|)^{-1} \} | > 2 | \overline{x}| .      \eqno(4.16)$$
Based on $(4.16)$, we find that $\overline{x} \in E_{1}^{\ast}$ and
\begin{center}
$ f(x)g^{\ast}(\overline{y})h^{\ast}(\overline{z})|a \overline{x}+ b \overline{y} + c \overline{z}| > \alpha +\varepsilon $  \ for all $x\in E_{1}$.
\end{center}
Then
$$g^{\ast}(\overline{y}) >  (\alpha + \frac{\varepsilon}{2}) (\displaystyle{\inf_{x\in E_{1}}}f(x) h^{\ast}(\overline{z}) |a \overline{x}+ b \overline{y} + c \overline{z}|)^{-1}.$$
Applying the property of symmetric rearrangement again we get
$$|E_{2}|\equiv | \{y: g(y) >(\alpha + \frac{\varepsilon}{2})  (\displaystyle{\inf_{x \in E_{1}}}f(x) h^{\ast}(\overline{z}) |a \overline{x}+ b \overline{y} + c \overline{z}|   )^{-1} \} | > 2 | \overline{y}| .  \eqno(4.17)$$
Based on it we have $\overline{y} \in E_{2}^{\ast}$  and for any $y \in E_{2}$,
$$\displaystyle{\inf_{x \in E_{1}}}f(x) g(y) h^{\ast}(\overline{z}) |a \overline{x}+ b \overline{y} + c \overline{z}|   >\alpha + \frac{\varepsilon}{2}.$$
Obviously,
$$ \displaystyle{\inf_{x \in E_{1}}}f(x)  \displaystyle{\inf_{y \in E_{2}}} g(y)  h^{\ast}(\overline{z}) |a \overline{x}+ b \overline{y} + c \overline{z}|   > \alpha + \frac{\varepsilon}{3}. $$
It follows from
 $$ h^{\ast}(\overline{z})  > (\alpha + \frac{\varepsilon}{3}) (\displaystyle{\inf_{x \in E_{1}}}f(x) \displaystyle{\inf_{y \in E_{2}}}g(y) |a \overline{x}+ b \overline{y} + c \overline{z}|  )^{-1}$$
and the property of decreasing rearrangement once more that
$$|E_{3}|\equiv | \{y: h(z) > (\alpha + \frac{\varepsilon}{3} ) (\displaystyle{\inf_{x \in E_{1}}}f(x) \displaystyle{\inf_{y \in E_{2}}}g(y) |a \overline{x}+ b \overline{y} + c \overline{z}|  )^{-1} \} | > 2 | \overline{z}| ,   \eqno(4.18)$$
so $\overline{z} \in E_{3}^{\ast}$ .
From (4.18)  we get  for any $x \in E_{1}$, $y \in E_{2}$, $z \in E_{3}$,
$$f(x)g(y)h(z) > (\alpha + \frac{\varepsilon}{3} )  (|a \overline{x}+ b \overline{y} + c \overline{z}|  )^{-1},$$
which implies
$$\displaystyle{\sup_{x\in E_{1}, y\in E_{2}, z\in E_{3}}} f(x)g(y)h(z) |a x+ b y + c z|  \geq  (\alpha + \frac{\varepsilon}{3} )  (|a \overline{x}+ b \overline{y} + c \overline{z}|  )^{-1} \displaystyle{\sup_{x\in E_{1}, y\in E_{2}, z\in E_{3}}} |a x+ b y + c z|.$$
Therefore,
\begin{align*}
\displaystyle{\sup_{x, y, z}}  f(x)g(y)h(z)|a x+ b y + c z|
 &\geq \displaystyle{\sup_{x\in E_{1}, y\in E_{2}, z\in E_{3}}} f(x)g(y)h(z)|a x+ b y + c z| \\
&\geq (\alpha + \frac{\varepsilon}{3} )  (|a \overline{x}+ b \overline{y} + c \overline{z}|)^{-1}   \displaystyle{\sup_{x\in E_{1}, y\in E_{2}, z\in E_{3}}} |a x+ b y + c z|.
\end{align*}

Lemma 4.7 gives that
$$ \displaystyle{\sup_{x\in E_{1}, y\in E_{2}, z\in E_{3}}} |a x+ b y + c z|  \geq \displaystyle{\sup_{x\in E_{1}^{\ast}, y\in E_{2}^{\ast}, z\in E_{3}^{\ast}} } |a x+ b y + c z|
\geq  |a \overline{x}+ b \overline{y} + c \overline{z}|,$$
since
$\overline{x} \in E_{1}^{\ast}$, $ \overline{y} \in E_{2}^{\ast}$,  $\overline{z} \in E_{3}^{\ast}$ .
Hence
$$(\alpha + \frac{\varepsilon}{3} )  (|a \overline{x}+ b \overline{y} + c \overline{z}|)^{-1}   \displaystyle{\sup_{x\in E_{1}, y\in E_{2}, z\in E_{3}}}|a x+ b y + c z|  \geq \alpha + \frac{\varepsilon}{3} >  \alpha , $$
which is a contradiction.  This completes the proof of Lemma 4.8.

\end{proof}

\bigskip

It follows from Lemma 4.8 and the rearrangement property
\begin{center}
$(f^{p})^{\ast}= (f^{\ast})^{p}$, \ \ for \ $0<p<\infty$
\end{center}
that  that we have the following corollary.

\begin{corollary}
Let $f_{j}$ be defined on $\mathbb{R}$, $j=1, \dots, n+1$. Then for any $1 \leq i \leq n$,
$$\displaystyle{\sup_{y_{j}}} \prod_{j=1}^{n+1} f_{j}^{\ast i}(y_{j})  \det( y_{1}, \dots, y_{n+1})^{\frac{n+1}{p}}
\leq   \displaystyle{\sup_{y_{j}}}  \prod_{j=1}^{n+1} f_{j}(y_{j})  \det(y_{1}, \dots, y_{n+1})^{\frac{n+1}{p}}.  \eqno(4.19)  $$
Furthermore,
$$\displaystyle{\sup_{y_{j}}} \prod_{j=1}^{n+1} \mathcal{R}_{n} \mathcal{R}_{n-1} \dots \mathcal{R}_{1} f_{j}(y_{j})  \det( y_{1}, \dots, y_{n+1})^{\frac{n+1}{p}}
\leq   \displaystyle{\sup_{y_{j}}}  \prod_{j=1}^{n+1} f_{j}(y_{j}) \det(y_{1}, \dots, y_{n+1})^{\frac{n+1}{p}}. $$
\end{corollary}

This is because  for each $i$,
$\det(y_{1}, \dots, y_{n+1})$ is the linear combination of $y_{1i}, \dots,y_{(n+1)i}$,
where $y_{ki}$  is the $ i$-th coordinate of $y_{k}\in \mathbb{R}^{n},  1 \leq k \leq n+1$.

\bigskip

We now turn to study the optimiser for multilinear form inequality  (4.11).
By the above, we need only look amongst the class of all functions which are decreasing and symmetric separately in each coordinate variable.

\bigskip

\hspace{-13pt}{\it Proof of Theorem 4.6}\quad
Let $\mathcal{S}$ be the stereographic projection from $\mathbb{R}^{n}$ to the nothern hemisphere  $\mathbb{S}_{+}^{n}$  with
$$\mathcal{S}(x)=(\frac{x_{1}}{\sqrt{1+|x|^{2}}},  \dots, \frac{x_{n}}{\sqrt{1+|x|^{2}}}, \frac{1}{\sqrt{1+|x|^{2}}}),$$
where $ x=( x_{1}, \dots, x_{n}) \in \mathbb{R}^{n}$.
So
$$\mathcal{S}^{-1}(s)= (\frac{s_{1}}{s_{n+1}}, \dots, \frac{s_{n}}{s_{n+1}} ).$$
For $f\in L^{p}(\mathbb{R}^{n})$, define
$$ (\mathcal{S}^{\ast}f)(s):= |J_{\mathcal{S}^{-1}}(s)|^{1/p}  f(\mathcal{S}^{-1}(s))  \eqno(4.20)$$
where $J_{\mathcal{S}^{-1}}$ is the Jacobian determinant of the map $\mathcal{S}^{-1}$,
$$ |J_{\mathcal{S}^{-1}}(s)| = (\frac{1}{s_{n+1}})^{n+1}= (1+|\mathcal{S}^{-1}(s)|^{2})^{\frac{n+1}{2}}. \eqno(4.21)$$
Denote $ \mathcal{S}^{\ast}f_{j}$  by $F_{j}$, and let $y_{j} =  \mathcal{S}^{-1} (s_{j})$,  then we have
\begin{align*}
 \det(y_{1}, \dots, y_{n+1}) &=  \prod_{j=1}^{n+1} ( 1+|y_{j}|^{2})^{1/2}  \det(s_{1}, \dots, s_{n+1})  \\
&=  \prod_{j=1}^{n+1}   |J_{\mathcal{S}^{-1}}(s_{j})| ^{ \frac{1}{n+1} }   \det(s_{1}, \dots, s_{n+1}),
\end{align*}
where $ \det(s_{1}, \dots, s_{n+1})$ is the absolute value of the  determinant of the matrix $(s_{1}, \dots, s_{n+1})_{(n+1) \times (n+1)}$.
Together with (4.20) and (4.21), we also have the conformally  invariant property of  multilinear geometric inequality  (4.11) as follows.
\begin{align*}
\displaystyle{\sup_{s_{j}\in \mathbb{S}_{+}^{n}}}  \prod_{j=1}^{n+1} F_{j}(s_{j}) \det(s_{1}, \dots, s_{n+1})^{\frac{n+1}{p}}
&= \displaystyle{\sup_{s_{j}\in \mathbb{S}_{+}^{n}}} |J_{\mathcal{S}^{-1}}(s_{j})|^{1/p}  f_{j}(\mathcal{S}^{-1}(s_{j}))  \det(s_{1}, \dots, s_{n+1})^{\frac{n+1}{p}} \\
&= \displaystyle{\sup_{y_{j}\in \mathbb{R}^{n}}}  \prod_{j=1}^{n+1} f_{j}(y_{j}) \det(y_{1}, \dots, y_{n+1})^{\frac{n+1}{p}},
\end{align*}
and the invariance of $L^{p}$ norm: for every $j$
$$\|f_{j}\|_{L^{p}(\mathbb{R}^{n})} = (\int_{\mathbb{R}^{n}} |f(x)|^{p} dx )^{1/p}
= (\int_{\mathbb{S}_{+}^{n}} |f(\mathcal{S}^{-1}(s_{j}))|^{p} |J_{\mathcal{S}^{-1}}(s_{j})| ds_{j} )^{1/p}
= (\int_{\mathbb{S}_{+}^{n}} |F_{j}(s_{j})|^{p} ds_{j})^{1/p}. $$

\medskip

For $f\in L^{p}(\mathbb{R}^{n})$, pick  $\alpha$ which is not a rational multiple of $\pi$.  For $1\leq i \leq n$,
we define $U_{\alpha}^{i}: \mathbb{S}_{+}^{n} \rightarrow \mathbb{S}_{+}^{n}$ be a rotation of the sphere $\mathbb{S}^{n}$  by angle $\alpha$  which
 keeps the other basis vectors fixed except the $i$-th and $(n+1)$-th vectors.
If the point after rotation  is  in the southern hemisphere,  we then send the point to its antipodal point in $\mathbb{S}_{+}^{n}$.
For $ F \in L^{p} (\mathbb{S}_{+}^{n})$, define
 $$(( U_{\alpha}^{i})^{\ast} F)(s):= |J_{ (U_{\alpha}^{i})^{-1}}(s)|^{\frac{1}{p} } F ( (U_{\alpha}^{i})^{-1}s ) =  F((U_{\alpha}^{i})^{-1}s).  $$
With the same $\mathcal{S}^{\ast}$ in (4.20), we consider the new function $(\mathcal{S}^{\ast})^{-1} (U_{\alpha}^{i} )^{\ast} \mathcal{S}^{\ast} f$.

\bigskip

In brief we denote this new function $(\mathcal{S}^{\ast})^{-1} (U_{\alpha}^{i} )^{\ast} \mathcal{S}^{\ast} f$
by  $ \mathcal{U}_{\alpha}^{i} f$.
For any $f\in  L^{p} (\mathbb{R}^{n})$,  we define a sequence  $\{f^{k}\}$  as in [4]  as follows,
\begin{center}
$f^{0}= f, \ f^{1} =\mathcal{R}_{n}  \mathcal{R}_{n-1} \dots  \mathcal{R}_{1}  \mathcal{U}_{\alpha}^{1} f$,  \  $f^{2} =  \mathcal{R}_{1} \mathcal{R}_{n}  \dots \mathcal{R}_{2}  \mathcal{U}_{\alpha}^{2} f^{1}$,
\end{center}
$$f^{3} =   \mathcal{R}_{2}   \mathcal{R}_{1}  \mathcal{R}_{n}  \dots \mathcal{R}_{3}  \mathcal{U}_{\alpha}^{3} f^{2},  \cdots,  f^{n+1} = \mathcal{R}_{n}  \dots \mathcal{R}_{1}  \mathcal{U}_{\alpha}^{1} f^{n} \cdots$$
Note that  $\mathcal{U}_{\alpha}^{i}$  and  $\mathcal{R}_{n}  \mathcal{R}_{n-1}  \dots \mathcal{R}_{1}$ are norm-preserving.
It follows from the proof of Theorem 8 in \cite{Valdimarsson} that
for any $f\in  L^{p} (\mathbb{R}^{n}) $,  we have
$\{f^{k}\}$ converges to $h_{f}$  in $ L^{p}$ norm, where
$$h_{f}= c h, \ \ \ \ h (x)=(\frac{1}{1+|x|^{2}})^{\frac{n+1}{2p}}$$
and $c$ is the constant such that  $\|f\|_{p}= \|c  h\|_{p}$.
So
$$c= |\mathbb{S}^{n}_{+}|^{-\frac{1}{p}} \|f\|_{p} .$$

\medskip

Here we will only sketch the argument, mainly using the competing symmetries in one dimension.
First  it is enough to consider the bounded functions that vanish outside a bounded set which are dense in $L^{p}$,
so there exists a constant $C$ such that $f(x) \leq Ch_{f}(x)$.
Note that $\mathcal{R}_{j}$ and $\mathcal{U}_{\alpha}^{j}$ are order-preserving,
then we have $f^{k}(x) \leq Ch_{f}(x)$ for every $f^{k}$ and all $x$.
By Helly's selection principle we can find a subsequence $f^{k_{l}}$ such that $f^{k_{l}}$  converges to some $g$ almost everywhere as $l \rightarrow \infty$.
The dominated convergence theorem implies that $g \in L^{p}$.
We define
$$A:= \inf\limits_{n} \|h_{f}-f^{k}\|_{p}= \lim\limits_{n} \|h_{f}-f^{k}\|_{p},$$
this is because $\|h_{f}-f^{k}\|_{p}$ decreases monotonically which follows from
the property
$$\|\mathcal{R}_{j}f-\mathcal{R}_{j}g\|_{p}\leq \|f-g\|_{p}, \ \ \|\mathcal{U}_{\alpha}^{j}f-\mathcal{U}_{\alpha}^{j}g\|_{p}= \|f-g\|_{p}  \eqno(4.22)$$
and the invariance of $h_{f}$ under each $\mathcal{R}_{j}$ and $\mathcal{U}_{\alpha}^{j}$. \\
Applying  these properties again gives  that
$$\begin{array}{ll}
A&= \lim\limits_{n} \|h_{f}-f^{k_{l}+1}\|_{p} \\
&= \|h_{f}- \mathcal{R}_{n}  \mathcal{R}_{n-1} \dots  \mathcal{R}_{1}  \mathcal{U}_{\alpha}^{1}g\|_{p} \\
&= \|\mathcal{R}_{n}  \mathcal{R}_{n-1} \dots  \mathcal{R}_{1}  \mathcal{U}_{\alpha}^{1}h_{f}-
 \mathcal{R}_{n}  \mathcal{R}_{n-1} \dots  \mathcal{R}_{1}  \mathcal{U}_{\alpha}^{1}g\|_{p} \\
&\leq  \| \mathcal{U}_{\alpha}^{1}h_{f}-  \mathcal{U}_{\alpha}^{1}g\|_{p} \\
&=  \|h_{f}-g\|_{p} =A,
\end{array}\eqno(4.23)$$
then we must have equality everywhere
$$\begin{array}{ll}
\|h_{f}- \mathcal{R}_{n}  \mathcal{R}_{n-1} \dots  \mathcal{R}_{1}  \mathcal{U}_{\alpha}^{1}g\|_{p}
&= \| h_{f}-\mathcal{R}_{n-1} \dots  \mathcal{R}_{1}  \mathcal{U}_{\alpha}^{1}g\|_{p}  \\
&= \dots   \\
&=\| h_{f}-\mathcal{R}_{1}  \mathcal{U}_{\alpha}^{1}g\|_{p}   \\
&=\| h_{f}- \mathcal{U}_{\alpha}^{1}g\|_{p}
\end{array}\eqno(4.24)$$
which implies  (see Theorem 3.5 of \cite{Lieb-Loss})
$$ \mathcal{R}_{n}  \mathcal{R}_{n-1} \dots  \mathcal{R}_{1}  \mathcal{U}_{\alpha}^{1}g
= \mathcal{R}_{n-1} \dots  \mathcal{R}_{1}  \mathcal{U}_{\alpha}^{1}g
= \dots = \mathcal{R}_{1}  \mathcal{U}_{\alpha}^{1}g
=  \mathcal{U}_{\alpha}^{1}g  \eqno(4.25)$$
It turns out that  $\mathcal{R}_{1}  \mathcal{U}_{\alpha}^{1}g=  \mathcal{U}_{\alpha}^{1}g$  and  $\mathcal{R}_{1}g=g$ imply
$ \mathcal{U}_{2 \alpha}^{1}g=g$ which shows
$\mathcal{S}^{\ast} g$ is invariant under the rotation through an angle $2 \alpha$
which keeps the other basis vectors fixed except the $1$-th and $(n+1)$-th ones.
In particular, $2 \alpha$ is an irrational multiple of $\pi$.
Therefore,
for any fixed $s_{2}, \dots, s_{n}$, $(\mathcal{S}^{\ast} g)(\cdot,  s_{2}, \dots, s_{n}, \cdot)$  is a constant.
Also we have
$$\mathcal{R}_{1}  \mathcal{U}_{\alpha}^{1}g=   \mathcal{U}_{\alpha}^{1}g=g.  \eqno(4.26) $$

\medskip

Similarly, if we replace $f^{k_{l}+1}$ in (4.23) by $f^{k_{l}+2}$,  together with (4.25)-(4.26) and Theorem 3.5 of \cite{Lieb-Loss}
 we have
$$ \mathcal{R}_{1}\mathcal{R}_{n} \dots  \mathcal{R}_{2} \mathcal{U}_{\alpha}^{2} g
= \mathcal{R}_{n} \dots  \mathcal{R}_{2} \mathcal{U}_{\alpha}^{2}  g
= \dots = \mathcal{R}_{2}  \mathcal{U}_{\alpha}^{2} g
=  \mathcal{U}_{\alpha}^{2} g.  \eqno(4.27)$$
From  $\mathcal{R}_{2}  \mathcal{U}_{\alpha}^{2}g=   \mathcal{U}_{\alpha}^{2}g$  and  $\mathcal{R}_{2}g=g$ we obtain  that
$ \mathcal{U}_{2 \alpha}^{2}g=g$ which shows
$\mathcal{S}^{\ast} g$ is invariant under the rotation through an angle $2 \alpha$
which keeps the other basis vectors fixed except the $2$-th and $(n+1)$-th ones.
So for any fixed $s_{1},  s_{3}, \dots, s_{n}$, $(\mathcal{S}^{\ast} g)(s_{1}, \cdot, s_{3}, \dots, s_{n}, \cdot)$  is  a constant,
since  $2 \alpha$ is an irrational multiple of $\pi$.
Meanwhile we have
$$\mathcal{R}_{2}  \mathcal{U}_{\alpha}^{2}g=   \mathcal{U}_{\alpha}^{2}g=g.  \eqno(4.28) $$
So far based on the discussion above,  we've got  for any fixed $ s_{3}, \dots, s_{n}$, $(\mathcal{S}^{\ast} g)(\cdot, \cdot, s_{3}, \dots, s_{n}, \cdot)$  must be  a constant.

By induction we can obtain
$ \mathcal{S}^{\ast}g$ is a constant function on $\mathbb{S}_{+}^{n}$,  and thus the corresponding function $g$ on $\mathbb{R}^{n}$ is $C h_{f}$.
Note that  $\mathcal{R}_{j}$ and $\mathcal{U}_{\alpha}^{j}$ are norm-preserving,
so
$$\|g\|_{p}= \lim\limits_{n} \|f^{k_{l}}\|_{p}= \|f\|_{p}, $$
which gives $C=1$, $g=h_{f}$.
Therefore, the sequence $f^{k}$ converges to $h_{f}$  in $ L^{p}$ norm.

\medskip

It follows from Lemma 4.8,   inequality (4.19)  of  Corollary 4.9
and the invariance of the multilinear geometric inequality under stereographic projection that
$\displaystyle{\sup_{y_{j}}} \ \displaystyle{\prod_{j=1}^{n+1}} \ f_{j}^{k}  (y_{j}) \det(y_{1}, \dots,y_{n+1})^{\frac{n+1}{p}}$  decreases monotonically as $k$ grows.
 That is for all $k \in \mathbb{N}$,
$$   \displaystyle{\sup_{y_{j}}} \ \displaystyle{\prod_{j=1}^{n+1}} \ f_{j}^{k}  (y_{j}) \det(y_{1}, \dots,y_{n+1})^{\frac{n+1}{p}}  \geq  \displaystyle{\sup_{y_{j}}} \ \displaystyle{\prod_{j=1}^{n+1}} \ f_{j}^{k+1}  (y_{j}) \ \det(y_{1}, \dots,y_{n+1})^{\frac{n+1}{p}} .$$
Since $\{f_{j}^{k}\}$ converges to $ h_{f_{j}}$  in $ L^{p}$  norm , $1 \leq j \leq n+1$,  then there exist subsequences
$\{ f_{j}^{k_{l}}\}$ such that $f_{j}^{k_{l}} \rightarrow  h_{f_{j}}$  pointwise almost everywhere as $ l  \rightarrow \infty$.
From
$$ \displaystyle{\sup_{y_{j}}} \ \displaystyle{\prod_{j=1}^{n+1}} \ f_{j}^{k_{l}}  (y_{j}) \det(y_{1}, \dots,y_{n+1})^{\frac{n+1}{p}}  \leq \displaystyle{\sup_{y_{j}}} \ \displaystyle{\prod_{j=1}^{n+1}} \ f_{j}  (y_{j}) \ \det(y_{1}, \dots,y_{n+1})^{\frac{n+1}{p}}  < \infty  $$
for all $k_{l}$ together with the dominated convergence theorem it follows that
$$ \displaystyle{\prod_{j=1}^{n+1}} \ f_{j}^{k_{l}}  (y_{j}) \det(y_{1}, \dots,y_{n+1})^{\frac{n+1}{p}}  \xrightarrow{ weak^* }  \displaystyle{\prod_{j=1}^{n+1}} \ h_{f_{j}}(y_{j}) \det(y_{1}, \dots,y_{n+1})^{\frac{n+1}{p}},$$
in $ L^{\infty}(\mathbb{R}^{n}) \times \dots \times L^{\infty}(\mathbb{R}^{n})$  as $l  \rightarrow \infty$.   \\
Hence  by the $\mathrm{weak}^*$ lower semicontinuity of the $L^{\infty}$ norm ,
\begin{align*}
 \displaystyle{\sup_{y_{j}}}  \  \displaystyle{\prod_{j=1}^{n+1}} \  h_{f_{j}}(y_{j}) \det(y_{1}, \dots,y_{n+1})^{\frac{n+1}{p}}
&\leq   \displaystyle{\liminf_{l}} \ ( \displaystyle{\sup_{y_{j}}} \ \displaystyle{\prod_{j=1}^{n+1}} \ f_{j}^{k_{l}}  (y_{j}) \det(y_{1}, \dots,y_{n+1})^{\frac{n+1}{p}} \\
&=   \displaystyle{\inf_{l}} \ (  \displaystyle{\sup_{y_{j}}} \ \displaystyle{\prod_{j=1}^{n+1}} \ f_{j}^{k_{l}}  (y_{j}) \det(y_{1}, \dots,y_{n+1})^{\frac{n+1}{p}}.
\end{align*}
Combining this with the norm-preserving property $\| f_{j} \|_{p} = \|  f_{j}^{k} \|_{p}$ for every $k$  and  the decreasing  property of $\displaystyle{\sup_{y_{j}}} \ \displaystyle{\prod_{j=1}^{n+1}} \ f_{j}^{k}  (y_{j}) \ \det(y_{1},...,y_{n+1})^{\frac{n+1}{p}}$,
we get for all  $f_{j} \in  L^{p} (\mathbb{R}^{n})$  and every $k_{l}$,
\begin{align*}
\frac{\displaystyle{\sup_{y_{j}}} \ \displaystyle{\prod_{j=1}^{n+1}} \ f_{j}  (y_{j}) \det(y_{1}, \dots,y_{n+1})^{\frac{n+1}{p}}}{ \displaystyle{\prod_{j=1}^{n+1}} \|f_{j}\|_{p}  }
&\geq \frac{\displaystyle{\sup_{y_{j}}} \ \displaystyle{\prod_{j=1}^{n+1}} \ f_{j}^{k_{l}}  (y_{j}) \det(y_{1}, \dots,y_{n+1})^{\frac{n+1}{p}}}{ \displaystyle{\prod_{j=1}^{n+1}} \|f_{j}^{k_{l}}\|_{p}  }  \\
&\geq  \frac{\displaystyle{\sup_{y_{j}}} \ \displaystyle{\prod_{j=1}^{n+1}} \   h_{f_{j}}(y_{j}) \det(y_{1}, \dots,y_{n+1})^{\frac{n+1}{p}}}{  \prod\limits_{j=1}^{n+1} \|h_{f_{j}}\|_{p}   }
 \end{align*}
Obviously,
$$\frac{\displaystyle{\sup_{y_{j}}} \ \displaystyle{\prod_{j=1}^{n+1}} \   h_{f_{j}}(y_{j}) \det(y_{1}, \dots,y_{n+1})^{\frac{n+1}{p}}}{ \prod\limits_{j=1}^{n+1} \|h_{f_{j}}\|_{p} }
= \frac{\displaystyle{\sup_{y_{j}}} \ \displaystyle{\prod_{j=1}^{n+1}} \   h(y_{j}) \det(y_{1}, \dots,y_{n+1})^{\frac{n+1}{p}}}{ \|h\|_{p}^{n+1}  }.$$

$\hfill\Box$

\bigskip

Based on Theorem 4.6  and the conformal invariance under the stereographic projection from $\mathbb{R}^{n}$ to $\mathbb{S}_{+}^{n}$,
the geometric inequality (4.11) has the conformally equivalent form in $\mathbb{S}^{n}$ space.

\medskip

\begin{theorem}
For $0 < p < \infty$, let $F_{j}$  be  nonnegative functions in $ L^{p}(\mathbb{S}^{n})$. Then
$$  \displaystyle{\prod_{j=1}^{n+1}} \|F_{j}\|_{L^{p}(\mathbb{S}^{n})} \leq B_{p,n} \
\displaystyle{\sup_{s_{j}\in \mathbb{S}^{n} }} \ \displaystyle{\prod_{j=1}^{n+1}} \ F_{j}(s_{j}) \det(s_{1}, \dots, s_{n+1})^{\frac{n+1}{p}},  \eqno(4.29)$$
where $\det(s_{1}, \dots, s_{n+1})$ is  the absolute value of  the determinant of the matrix $(s_{1}, \dots, s_{n+1})_{(n+1) \times (n+1)}$.
The best constant $B_{p,n}$ is obtained when $F_{j}(s_{j})$ are constant, and the  corresponding  $B_{p,n}= |\mathbb{S}^{n}|^{\frac{n+1}{p}}$.
\end{theorem}

\begin{proof}
From  Theorem 4.6  and and the conformal invariance of (4.11) under the stereographic projection from $\mathbb{R}^{n}$ to $\mathbb{S}_{+}^{n}$,
we obtain for nonnegative functions $F_{j} \in L^{p}(\mathbb{S}_{+}^{n})$,
$$  \displaystyle{\prod_{j=1}^{n+1}} \|F_{j}\|_{L^{p}(\mathbb{S}_{+}^{n})} \leq C_{p,n} \
\displaystyle{\sup_{s_{j}\in \mathbb{S}_{+}^{n} }} \ \displaystyle{\prod_{j=1}^{n+1}} \ F_{j}(s_{j}) \det(s_{1},..., s_{n+1})^{\frac{n+1}{p}}  \eqno(4.30)$$
holds.  The best constant $C_{p,n}$ is  obtained when $F_{j}(s_{j})$ are constant, and the corresponding
$C_{p,n}= |\mathbb{S}_{+}^{n}|^{\frac{n+1}{p}}=(\frac{1}{2}|\mathbb{S}^{n}| )^{\frac{n+1}{p}}$.
Note that
$$\displaystyle{\sup_{s_{j} \in \mathbb{S}_{+}^{n} }} \ \det(s_{1}, \dots,s_{n+1})=\displaystyle{\sup_{s_{j} \in \mathbb{S}^{n} }} \ \det(s_{1}, \dots,s_{n+1})=1.$$

Let $F_{j}$  be  nonnegative functions in $ L^{p}(\mathbb{S}^{n})$. We define
$$\overline{F}_{j}(s_{j})=\max \{F_{j}(s_{j}), F_{j}(\overline{s}_{j})\},$$
where $\overline{s}_{j}$ is the antipodal point of $s_{j}$, $1 \leq j \leq n+1$,
$s_{j} \in \mathbb{S}_{+}^{n}$. \\
Then  $\overline{F}_{j} \in  L^{p}(\mathbb{S}_{+}^{n})$, and
$$\displaystyle{\sup_{s_{j} \in \mathbb{S}_{+}^{n} }} \ \displaystyle{\prod_{j=1}^{n+1}} \ \overline{F}_{j}(s_{j}) \det(s_{1}, \dots,s_{n+1})^{\frac{n+1}{p}}
=
\displaystyle{\sup_{s_{j} \in \mathbb{S}^{n} }} \ \displaystyle{\prod_{j=1}^{n+1}} \ F_{j}(s_{j}) \det(s_{1}, \dots,s_{n+1})^{\frac{n+1}{p}}, \eqno(4.31)$$
this is because for any $s_{j} \in \mathbb{S}_{+}^{n}$,
$$\det(s_{1}, \dots, \overline{s}_{j}, \dots,s_{n+1})= \det(s_{1}, \dots,s_{j}, \dots,s_{n+1}).$$
Besides,
$$2 \|\overline{F}_{j}\|_{L^{p}(\mathbb{S}_{+}^{n})}^{p}
\geq  \int_{\mathbb{S}_{+}^{n}} (F_{j}(s_{j}))^{p}  d s_{j} +    \int_{\mathbb{S}_{+}^{n}} (F_{j}(\overline{s}_{j}))^{p} d s_{j}
=\|F_{j}\|_{L^{p}(\mathbb{S}^{n})}^{p}. $$
Thus  for each $j$
$$\|\overline{F}_{j}\|_{L^{p}(\mathbb{S}_{+}^{n})} \geq  2^{-\frac{1}{p}} \|F_{j}\|_{L^{p}(\mathbb{S}^{n})}. \eqno(4.32)$$
It follows from (4.30)-(4.32) that   for any nonnegative $F_{j} \in  L^{p}(\mathbb{S}^{n})$,
\begin{align*}
 \displaystyle{\prod_{j=1}^{n+1}} \|F_{j}\|_{L^{p}(\mathbb{S}^{n})}
&\leq 2^{\frac{n+1}{p}} \displaystyle{\prod_{j=1}^{n+1}}  \|\overline{F}_{j}\|_{L^{p}(\mathbb{S}_{+}^{n}) } \\
&\leq 2^{\frac{n+1}{p}} |\mathbb{S}_{+}^{n}|^{\frac{n+1}{p}} \
\displaystyle{\sup_{s_{j}\in \mathbb{S}_{+}^{n} }} \ \displaystyle{\prod_{j=1}^{n+1}} \ \overline{F}(s_{j}) \det(s_{1}, \dots, s_{n+1})^{\frac{n+1}{p}} \\
&= |\mathbb{S}^{n}|^{\frac{n+1}{p}} \ \displaystyle{\sup_{s_{j} \in \mathbb{S}^{n} }} \ \displaystyle{\prod_{j=1}^{n+1}} \ F_{j}  (s_{j}) \det(s_{1}, \dots,s_{n+1})^{\frac{n+1}{p}},
\end{align*}
which proves (4.29).

To show that  $|\mathbb{S}^{n}|^{\frac{n+1}{p}}$ is the best constant in (4.29),
suppose for a contradiction that
 $F_{j}(s_{j}) \in L^{p}(\mathbb{S}^{n})$ is an optimiser for (4.29) that satisfies
$$\frac{\displaystyle{\sup_{s_{j} \in \mathbb{S}^{n} }} \ \displaystyle{\prod_{j=1}^{n+1}} \ F_{j}  (s_{j}) \det(s_{1}, \dots,s_{n+1})^{\frac{n+1}{p}}}{ \displaystyle{\prod_{j=1}^{n+1}} \|F_{j}\|_{L^{p}(\mathbb{S}^{n}) } }
< \frac{\displaystyle{\sup_{s_{j} \in \mathbb{S}^{n} }} \ \det(s_{1}, \dots,s_{n+1})^{\frac{n+1}{p}}}{ |\mathbb{S}^{n}|^{\frac{1}{p}} \cdots |\mathbb{S}^{n}|^{\frac{1}{p}}  }
= |\mathbb{S}^{n}|^{-\frac{n+1}{p}} .$$
Then by (4.31)-(4.32)  we find $\overline{F}_{j}(s_{j})$ defined as above satisfying
\begin{align*}
\frac{\displaystyle{\sup_{s_{j}\in \mathbb{S}_{+}^{n} }} \ \displaystyle{\prod_{j=1}^{n+1}} \ \overline{F}_{j}(s_{j}) \det(s_{1}, \dots, s_{n+1})^{\frac{n+1}{p}}}{\displaystyle{\prod_{j=1}^{n+1}}  \|\overline{F}_{j}\|_{L^{p}(\mathbb{S}_{+}^{n}) }}
&\leq  \frac{\displaystyle{\sup_{s_{j} \in \mathbb{S}^{n} }} \
\displaystyle{\prod_{j=1}^{n+1}} \ F_{j}  (s_{j}) \det(s_{1}, \dots,s_{n+1})^{\frac{n+1}{p}}}{2^{-\frac{n+1}{p}} \displaystyle{\prod_{j=1}^{n+1}} \|F_{j}\|_{L^{p}(\mathbb{S}^{n}) } } \\
&< 2^{\frac{n+1}{p}} |\mathbb{S}^{n}|^{-\frac{n+1}{p}} = |\mathbb{S}_{+}^{n}|^{-\frac{n+1}{p}}.
\end{align*}
This is  in  contradiction to the best constant in (4.30).
Hence the best constant  $B_{p,n}$ in (4.29) is $|\mathbb{S}^{n}|^{\frac{n+1}{p}}$.

\end{proof}

\medskip

Also the geometric  inequality (4.11) has the conformally equivalent form in
$\mathbb{H}^{n}$ space as shown in the following theorem.

\medskip

\begin{theorem}
For $0 < p < \infty$, let $F_{j}$  be  nonnegative functions in $ L^{p}(\mathbb{H}^{n})$. Then
$$  \displaystyle{\prod_{j=1}^{n+1}} \|F_{j}\|_{L^{p}(\mathbb{H}^{n})} \leq E_{p,n} \
\displaystyle{\sup_{q_{j}\in \mathbb{H}^{n}}} \ \displaystyle{\prod_{j=1}^{n+1}} \ F_{j}(q_{j}) \det(q_{1}, \dots,q_{n+1})^{\frac{n+1}{p}},  \eqno(4.33)$$
where $\det(q_{1}, \dots,q_{n+1})$ is the absolute value of the  determinant of the matrix $(q_{1}, \dots, q_{n+1})_{(n+1) \times (n+1)}$.

\end{theorem}

\begin{proof}
Consider the stereographic projection $\mathcal{H}$ which is a conformal transformation
from the unit disk $D^{n}$ in  $\mathbb{R}^{n}$  to $\mathbb{H}_{+}^{n}$ as
$$\mathcal{H}(x)=(\frac{x_{1}}{\sqrt{1-|x|^{2}}}, \dots,\frac{x_{n}}{\sqrt{1-|x|^{2}}}, \frac{1}{\sqrt{1-|x|^{2}}}),$$
where $x= (x_{1}, \dots, x_{n}) \in \mathbb{R}^{n}$.
So
$$\mathcal{H}^{-1}(q)= (\frac{q_{1}}{q_{n+1}}, \dots,\frac{q_{n}}{q_{n+1}}).$$
Here the Jacobian determinant of the map $\mathcal{H}^{-1}$ is
$$|J_{\mathcal{H}^{-1}(q)}|= (1- |\mathcal{H}^{-1}(q)|^{2})^{\frac{n+1}{2}}.$$
Let $y_{j}=\mathcal{H}^{-1}(q_{j})$,
then we have
\begin{align*}
 \det(y_{1}, \dots, y_{n+1}) &=  \prod_{j=1}^{n+1} ( 1-|y_{j}|^{2})^{1/2}  \det(q_{1}, \dots, q_{n+1})  \\
&=    \prod_{j=1}^{n+1}  |J_{\mathcal{H}^{-1}}(q_{j})| ^{ \frac{1}{n+1} }  \det(q_{1}, \dots, q_{n+1}).
\end{align*}
Define $F_{j}(q_{j}):= |J_{\mathcal{H}^{-1}}(q_{j})|^{1/p} f_{j}(\mathcal{H}^{-1}(q_{j}))$,
then from above we easily get the conformal invariance as follows.
\begin{align*}
\displaystyle{\sup_{q_{j}\in \mathbb{H}_{+}^{n}}} \ \displaystyle{\prod_{j=1}^{n+1}} \ F_{j}(q_{j}) \det(q_{1}, \dots,q_{n+1})^{\frac{n+1}{p}}
&= \displaystyle{\sup_{q_{j}\in \mathbb{H}_{+}^{n}}} \ \displaystyle{\prod_{j=1}^{n+1}} \ |J_{\mathcal{H}^{-1}}(q_{j})|^{1/p} f_{j}(\mathcal{H}^{-1}(q_{j})) \det(q_{1}, \dots,q_{n+1})^{\frac{n+1}{p}} \\
&= \displaystyle{\sup_{y_{j}\in D^{n}}} \ \displaystyle{\prod_{j=1}^{n+1}} \ f_{j}(y_{j}) \det(y_{1}, \dots,y_{n+1})^{\frac{n+1}{p}},
\end{align*}
and for every $j$
\begin{align*}
\|f_{j}\|_{L^{p}(D^{n})} &= (\int_{D^{n}} |f_{j}(y_{j})|^{p} dy_{j})^{1/p} \\
&= (\int_{\mathbb{H}_{+}^{n}} |f_{j}(\mathcal{H}^{-1}(q_{j}))|^{p} |J_{\mathcal{H}^{-1}(q)}| dq_{j})^{1/p} \\
&= (\int_{\mathbb{H}_{+}^{n}} |F_{j}(q_{j})|^{p} dq_{j})^{1/p}
= \|F_{j}\|_{L^{p}(\mathbb{H}_{+}^{n})}.
\end{align*}
Theorem 4.6 implies that
$$ \displaystyle{\prod_{j=1}^{n+1}} \|f_{j}\|_{L^{p}(D^{n})} \leq C_{p,n}
\displaystyle{\sup_{y_{j}\in D^{n}}} \ \displaystyle{\prod_{j=1}^{n+1}} \ f_{j}(y_{j}) \det(y_{1}, \dots,y_{n+1})^{\frac{n+1}{p}}, \eqno(4.34)$$
where $C_{p,n}= |\mathbb{S}_{+}^{n}|^{\frac{n+1}{p}}$.
Thus from the discussion above we have
$$  \displaystyle{\prod_{j=1}^{n+1}} \|F_{j}\|_{L^{p}(\mathbb{H}_{+}^{n})} \leq   |\mathbb{S}_{+}^{n}|^{\frac{n+1}{p}} \
\displaystyle{\sup_{q_{j}\in \mathbb{H}_{+}^{n}}} \ \displaystyle{\prod_{j=1}^{n+1}} \ F_{j}(q_{j}) \det(q_{1}, \dots,q_{n+1})^{\frac{n+1}{p}}.$$
Similarly to the arguments in the proof of Theorem 4.10,  we also have for any nonnegative $F_{j} \in  L^{p}(\mathbb{H}^{n})$,  $0 < p < \infty$,
there exists a finite constant $E_{p,n}= 2^{\frac{n+1}{p}}  |\mathbb{S}_{+}^{n}|^{\frac{n+1}{p}}$ such that
$$  \displaystyle{\prod_{j=1}^{n+1}} \|F_{j}\|_{L^{p}(\mathbb{H}^{n})} \leq  E_{p,n} \
\displaystyle{\sup_{q_{j}\in \mathbb{H}^{n}}} \ \displaystyle{\prod_{j=1}^{n+1}} \ F_{j}(q_{j}) \det(q_{1}, \dots,q_{n+1})^{\frac{n+1}{p}}. $$

\end{proof}

\medskip

However we do not know the extremal functions and the  best constant $C_{p,n}$ for inequality (4.34).
So a problem is to determine the optimisers  and the best constant $E_{p,n}$
for the  multilinear geometric inequality  (4.33) in hyperbolic space $\mathbb{H}^{n}$.

\bigskip

\hspace{-13pt}{\bf Acknowledgments.}\quad

I am  grateful to my supervisor Professor Carbery for his helpful advice and constructive suggestions.
I would like to thank Professor Carbery for repeatedly making discussion and revision on this paper.
This work was supported by the scholarship from China Scholarship Council.

\bibliographystyle{amsalpha}

\begin{thebibliography}{AcBeRu}

\bibitem{Gressman} P.T. Gressman, On multilinear determinant functionals, Proc. Amer. Math. Soc. 139 (2010) 2473-2484.\\

\bibitem{Beckner} W. Beckner, Geometric inequalities in Fourier analysis, Essays on Fourier Analysis in honor of Elias
M. Stein, Princeton University Press, 36-68, 1995. \\

\bibitem{Lieb-Loss}  H. Lieb and M. Loss, Analysis, American Mathematical Society, 2001.\\

\bibitem{Valdimarsson} S.I. Valdimarsson, A multilinear generalisation of the Hilbert transform and fractional
integration, Rev. Mat. Iberoam. 28 (2012) no.1, 25-55. \\


\end{thebibliography}

\end{document}